\documentclass[11pt]{article}
\usepackage{amsfonts,amsmath,bm, xspace}
\usepackage{multirow,graphicx, algorithm,algorithmic,rotating}
\usepackage{url}
\usepackage[round]{natbib}
\usepackage{fullpage}
\usepackage[small,bf]{caption}
\newtheorem{theorem}{Theorem}[section]
\newtheorem{lemma}[theorem]{Lemma}

\newtheorem{proposition}[theorem]{Proposition}

\newtheorem{conjecture}[theorem]{Conjecture}

\newtheorem{example}[subsection]{Example}

\def \endprf{\hfill {\vrule height6pt width6pt depth0pt}\medskip}
\newenvironment{proof}{\noindent {\bf Proof} }{\endprf\par}

\numberwithin{equation}{section}

\newcommand{\eat}[1]{}

\newcommand{\vct}[1]{\bm{#1}}
\newcommand{\mtx}[1]{\bm{#1}}

\newcommand{\wo}{\mathrm{wo}}
\renewcommand{\wr}{\mathrm{wr}}

\newcommand{\R}{\mathbb{R}}
\newcommand{\E}{\operatorname{\mathbb{E}}}

\newcommand{\diag}{\operatorname{diag}}
\newcommand{\trace}{\operatorname{trace}}

\newcommand{\minimize}{\mbox{minimize}}

\newcommand{\norm}[1]{\|#1\|}
\newcommand{\setof}[2]{\{ #1 \mid #2 \}}
\newcommand{\set[1]}{\{ #1 \}}

\newcommand{\eq}[1]{(\ref{eq:#1})}
\newcommand{\nn}{\nonumber}

\title{Beneath the valley of the noncommutative arithmetic-geometric mean inequality: conjectures, case-studies, and consequences}

\author{Benjamin Recht and Christopher R\'{e}\\
Computer Sciences Department, University of Wisconsin-Madison\\
1210 W Dayton St, Madison, WI 53706}

\date{February 2012}

\begin{document}

\maketitle

\vspace{-0.3in}

\begin{abstract}
Randomized algorithms that base iteration-level decisions on samples from some pool are ubiquitous in machine learning and optimization. Examples include stochastic gradient descent and randomized coordinate descent. This paper makes progress at theoretically evaluating the difference in performance between sampling with- and without-replacement in such algorithms.  Focusing on least means squares optimization, we formulate a \emph{noncommutative arithmetic-geometric mean inequality} that would prove that the expected convergence rate of without-replacement sampling is faster than that of with-replacement sampling.  We demonstrate that this inequality holds for many classes of random matrices and for some pathological examples as well.  We provide a deterministic worst-case bound on the gap between the discrepancy between the two sampling models, and explore some of the impediments to proving this inequality in full generality.   We detail the consequences of this inequality for stochastic gradient descent and the randomized Kaczmarz algorithm for solving linear systems.
\end{abstract}


\section{Introduction}

Randomized sequential algorithms abound in machine learning and optimization.  The most famous is the stochastic gradient method~\citep[see][]{Bottou98,BertsekasIncrementalSurvey,Nemirovski09,Shalev-Shwartz08}, but other popular methods include algorithms for alternating projections~\citep[see][]{Strohmer09,Leventhal10}, proximal point methods~\citep[see][]{Bertsekas11}, coordinate descent~\citep[see][]{Nesterov10} and derivative free optimization~\citep[see][]{Nesterov11,NemirovskiYudinBook}.  In all of these cases, an iterative procedure is derived where, at each iteration, an \emph{independent} sample from some distribution determines the action at the next stage.  This sample is selected \emph{with-replacement} from a pool of possible options.


In implementations of many of these methods, however, practitioners often choose to break the independence assumption.  For instance, in stochastic gradient descent, many implementations pass through each item exactly once in a random order (i.e., according to a random permutation).  In randomized coordinate descent, one can cycle over the coordinates in a random order.  These strategies, employing \emph{without-replacement} sampling, are often easier to implement efficiently, guarantee that every item in the data set is touched at least once, and often have better empirical performance than their with-replacement counterparts~\citep[see][]{Bottou09,RechtRe11,RechtReSIGMOD12}.

Unfortunately, the analysis of without-replacement sampling schemes are quite difficult.  The independence assumption underlying with-replacement sampling provides an elegant Markovian framework for analyzing incremental algorithms.  The iterates in without-replacement sampling are correlated, and studying them requires sophisticated probabilistic tools.  Consequently, most of the analyses without-replacement optimization assume that the iterations are assigned deterministically. Such deterministic orders might incur exponentially worse convergence rates than randomized methods~\citep{Nedic00}, and deterministic orders still require careful estimation of accumulated errors~\citep[see][]{Luo91,Tseng98}.  The goal of this paper is to make progress towards patching the discrepancy between theory and practice of without-replacement sampling in randomized algorithms.

In particular, in many cases, we  demonstrate that without-replacement sampling outperforms with-replacement sampling provided a \emph{noncommutative} version of the arithmetic-geometric mean inequality holds.  Namely, if $\mtx{A}_1,\ldots,\mtx{A}_n$ are a collection of $d\times d$ positive semidefinite matrices, we define the arithmetic and (symmetrized) geometric means to be
\begin{equation*}\label{eq:gm-def}
\mtx{M}_A:=\frac{1}{n}\sum_{i=1}^n \mtx{A}_i\,,~~~~\text{and}~~~~\mtx{M}_G:=\frac{1}{n!} \sum_{\sigma \in S_n} \mtx{A}_{\sigma(1)} \times \cdots \times \mtx{A}_{\sigma(n)}
\end{equation*}
where $S_n$ denotes the group of permutations.  Our conjecture is that the norm of $\mtx{M}_G$ is always less than the norm of $(\mtx{M}_A)^n$.  Assuming this inequality, we show that without-replacement sampling leads to faster convergence for both the least mean squares  and randomized Kaczmarz algorithms of~\citet{Strohmer09}.  

Using established work in matrix analysis, we show that these noncommutative arithmetic-geometric mean inequalities hold when there are only two matrices in the pool.  We also prove that the inequality is true when all of the matrices commute.  We demonstrate that if we don't symmetrize, there are deterministically ordered products of $n$ matrices whose norm exceeds $\|\mtx{M}_A\|^n$ by an exponential factor.  That is, symmetrization is necessary for the noncommutative arithmetic-geometric mean inequality to hold.

While we are unable to prove the noncommutative arithmetic-geometric mean inequality in full generality, we verify that it holds for many classes of~\emph{random} matrices.
Random matrices are, in some sense, the most interesting case for machine learning applications.  This is particularly evident in applications such as empirical risk minimization and online learning where the the data are conventionally assumed to be generated by some i.i.d random process.  In Section~\ref{sec:random}, we show that if $\mtx{A}_1,\ldots, \mtx{A}_n$ are generated i.i.d. from certain distributions, then the noncommutative arithmetic-geometric mean inequality holds in expectation with respect to the $\mtx{A}_i$.  Section~\ref{sec:random-ncamgm} assumes that $\mtx{A}_i = \mtx{Z}_i \mtx{Z}_i^T$ w.b.here $\mtx{Z}_i$ have independent entries, identically sampled from some symmetric distribution.  In Section~\ref{sec:sgd-consequences}, we analyze the random matrices that commonly arise in stochastic gradient descent and related algorithms, again proving that without-replacement sampling exhibits faster convergence than with-replacement sampling. We close with a discussion of other open conjectures that could impact machine learning theory, algorithms, and software.

\section{Sampling in incremental gradient descent}\label{sec:examples}

To illustrate how with- and without-replacement sampling methods differ in randomized optimization algorithms, we focus on one core algorithm, the Incremental Gradient Method (IGM).   Recall that the IGM minimizes the function
\begin{equation}\label{eq:glob-prob}
	\underset{\vct{x}}{\minimize}~f(\vct{x}) = \sum_{i=1}^n f_i(\vct{x})
\end{equation}
via the iteration
\begin{equation}\label{eq:igm-def}
	\vct{x}_k = \vct{x}_{k-1} - \gamma_{k} \nabla f_{i_k}(\vct{x}_{k-1})\,.
\end{equation}
Here, $\vct{x}_0$ is an initial starting vector, $\gamma_k$ are a sequence of nonnegative step sizes, and the indices $i_k$ are chosen using some (possibly deterministic) sampling scheme.   When $f$ is strongly convex, the IGM iteration converges to a near-optimal solution of~\eq{glob-prob} for any $\vct{x}_0$ under a variety of step-sizes protocols and sampling schemes including constant and diminishing step-sizes~\citep[see][]{Antstreicher00,BertsekasIncrementalSurvey,Nemirovski09}.  When the increments are selected uniformly at random at each iteration, IGM is equivalent to stochastic gradient descent.  We use the term IGM here to emphasize that we are studying many possible orderings of the increments.  In the next examples, we study the specialized case where the $f_i$ are quadratic and the IGM is equivalent to the \emph{least mean squares} algorithm of~\citet{WidrowHoffLMS}.

\subsection{One-dimensional Examples}
First consider the following toy one-dimensional least-squares problem 
\begin{equation}\label{eq:1dls}
	\underset{x}{\minimize}~\frac{1}{2} \sum_{i=1}^n (x- y_i)^2\,.
\end{equation}
where $y_i$ is a sequence of scalars with mean $\mu_y$ and variance $\sigma^2$.
Applying~\eq{igm-def} to~\eq{1dls} results in the iteration.
\[
	x_k = x_{k-1} - \gamma_{k} (x_{k-1}-y_{i_k})\,
\]
If we initialize the method with $x_0=0$ and take $n$ steps of incremental gradient with stepsize $\gamma_k=1/k$, we have
\[
	x_n = \frac{1}{n}\sum_{j=1}^n y_{i_j}
\]
where ${i_j}$ is the index drawn at iteration $j$.   If the steps are chosen using a without-replacement sampling scheme, $x_n =\mu_y$, the global minimum.    On the other hand, using with-replacement sampling, we will have
\[
	\E[ (x_n-\mu_y)^2] = \frac{\sigma^2}{n}\,,
\]
which is a positive mean square error.

Another toy example that further illustrates the discrepancy is the least-squares problem
\begin{equation*}\label{eq:1dls2}
	\underset{x}{\minimize}~\frac{1}{2} \sum_{i=1}^n \beta_i (x- y)^2
\end{equation*}
where $\beta_i$ are positive weights.  Here, $y$ is a scalar, and the global minimum is clearly $y$.  Let's consider the incremental gradient method with constant stepsize $\gamma_k=\gamma<\min \beta_i^{-1}$. Then after $n$ iterations we will have
\[
	|x_n - y| = |y| \prod_{j=1}^n (1-\gamma \beta_{i_j}) 
\]
If we perform without-replacement sampling, this error is given by
\[
	|x_n - y| =  |y| \prod_{i=1}^n (1-\gamma \beta_{i})\,.
\]
On the other hand, using with-replacement sampling yields
\[
	\E[|x_n - y|]= |y| \left(1-\frac{\gamma}{n}\sum_{i=1}^n \beta_{i}\right)^n \,.
\]
By the arithmetic-geometric mean inequality, we then have that the without-replacement sample is always closer to the optimal value in expectation.  This sort of discrepancy is not simply a feature of these toy examples.  We now demonstrate that similar behavior arises in multi-dimensional examples.

\subsection{IGM in more than one dimension}\label{sec:igm}

Now consider IGM in higher dimensions.  Let $\vct{x}_\star$ be a vector in $\R^d$ and set
\[
	y_i = \vct{a}_i^T \vct{x}_\star+ \omega_i~~~~\mbox{for}~i=1,\ldots, n
\]
where $\vct{a}_i\in\R^d$ are some test vectors and $\omega_i$ are i.i.d. Gaussian random variables with mean zero and variance $\rho^2$.

We want to compare with- vs without-replacement sampling for IGD on the cost function
\begin{equation}\label{eq:lsopt}
	\underset{\vct{x}}{\minimize}~\sum_{i=1}^n (\vct{a}_i^T\vct{x} - y_i)^2\,.
\end{equation}

Suppose we walk over $k$ steps of IGD with constant stepsize $\gamma$ and we access the terms $i_1,\ldots, i_k$ in that order.  Then we have
\[
\begin{aligned}
	\vct{x}_k = \vct{x}_{k-1} - \gamma \vct{a}_{i_k}(\vct{a}_{i_k}^T \vct{x}_{i_{k-1}}-y_{i_k}) = \left(I- \gamma \vct{a}_{i_k}\vct{a}_{i_k}^T\right)\vct{x}_{i_{k-1}} +\gamma \vct{a}_{i_k} y_{i_k}\,.
\end{aligned}
\]
Subtracting $x_\star$ from both sides of this equation then gives 
\begin{align}
\label{eq:basic-recursion} 	\vct{x}_k -\vct{x}_\star &= \left(\mtx{I}- \gamma \vct{a}_{i_k}\vct{a}_{i_k}^T\right)(\vct{x}_{k-1} -\vct{x}_\star) + \gamma \vct{a}_{i_k} \omega_{i_k}\\
\nonumber	&= \prod_{j=1}^k \left(\mtx{I}- \gamma \vct{a}_{i_j}\vct{a}_{i_j}^T\right)(\vct{x}_{0} -\vct{x}_\star) + \sum_{\ell=1}^k \prod_{k\geq j > \ell} \left(\mtx{I}-\gamma \vct{a}_{i_j} \vct{a}_{i_j}^T\right) \gamma \vct{a}_{i_\ell} \omega_{i_\ell}\,.
\end{align}
Here, the product notation means we multiply by the matrix with smallest index first, then left multiply by the matrix with the next index and so on up to the largest index.

Our goal is to estimate the risk after $k$ steps, namely $\E[\|\vct{x}_k - \vct{x}_\star\|^2]$, and demonstrate that this error is smaller for the without-replacement model.   The expectation is with respect to the IGM ordering and the noise sequence $\omega_i$. To simplify things a bit, we take a partial expectation with respect to $\omega_i$:
\begin{equation}\label{eq:main-sgd-exp}
\small
\E[\|\vct{x}_k -\vct{x}_\star\|^2] = \E\left[\left\| \prod_{j=1}^k \left(\mtx{I}- \gamma \vct{a}_{i_j}\vct{a}_{i_j}^T\right)(\vct{x}_{0} -\vct{x}_\star) \right\|^2\right]+\rho^2 \gamma^2\sum_{\ell=1}^k \E\left[\left\|\prod_{k\geq j > \ell} \left(\mtx{I}-\gamma \vct{a}_{i_j} \vct{a}_{i_j}^T\right) \vct{a}_{i_\ell}\right\|^2\right]
\end{equation}
In this case, we need to compare the expected value of matrix products under with or without-replacement sampling schemes in order to conclude which is better.  Is there a simple conjecture, analogous to the arithmetic-geometric mean inequality, that would guarantee without-replacement sampling is always better? 

\subsection{The Randomized Kaczmarz algorithm}

As another high-dimensional example in the same spirit, we consider the randomized Kaczmarz algorithm of~\citet{Strohmer09}. The Kaczmarz algorithm is used to solve the over-determined linear system  $\mtx{\Phi} \vct{x} = \vct{y}$. Here $\mtx{\Phi}$ is an $n\times d$ matrix with $n>d$ and we assume there exists an exact solution $\mtx{x}_\star$ satisfying $\mtx{\Phi}\vct{x}_\star = \vct{y}$.  Kaczmarz's method solves this system by alternating projections~\citep{Kaczmarz37} and was implemented in the earliest medical scanning devices~\citep{Hounsfield73}.  In computer tomography, this method is called the \emph{Algebraic Reconstruction Technique}~\citep{Herman80, Natterer86} or \emph{Projection onto Convex Sets}~\citep{Sezan87}.  

Kaczmarz's algorithm consists of iterations of the form
\begin{equation}\label{eq:main-kacz-exp}
	\vct{x}_{k+1} = \vct{x}_k + \frac{y_i - \vct{\phi}_{i_k}^T \vct{x}_k}{\|\vct{\phi}_{i_k}\|^2} \vct{\phi_{i_k}}\,.
\end{equation}
where the rows of $\Phi$ are accessed in some deterministic order.  This sequence can be interpreted as an incremental variant of Newton's method on the least squares cost function
\[
	\underset{\vct{x}}{\minimize}~\sum_{i=1}^n (\vct{\phi}_{i}^T \vct{x}_k-y_i)^2
\]
with step size equal to $1$~\citep[see][]{BertsekasNLP}.

Establishing the convergence rate of this method proved difficult in imaging science.  On the other hand,~\citet{Strohmer09} proposed a randomized variant of the Kaczmarz method, choosing the next iterate with-replacement with probability proportional to the norm of $\vct{\phi}_i$.  Strohmer and Vershynin established linear convergence rates for their iterative scheme.  Expanding out~\eq{main-kacz-exp} for $k$ iterations, we see that
\[
	\vct{x}_k - \vct{x}_\star = \prod_{j=1}^k \left(\mtx{I} - \frac{\vct{\phi}_{i_j}\vct{\phi}_{i_j}^T}{\|\vct{\phi}_{i_j}\|^2} \right) (\vct{x}_0 - \vct{x}_\star)\,.
\]
Let us suppose that we modify Strohmer and Vershynin's procedure to employ without-replacement sampling. After $k$ steps is the with-replacement or without-replacement model closer to the optimal solution?

\section{Conjectures concerning the norm of geometric and arithmetic means of positive definite matrices}
\eat{
What are the necessary ingredients required to generalize the simple one-dimensional examples to higher dimensions?  In particular, what is the appropriate matrix inequality which would guarantee that without-replacement sampling is superior to with-replacement sampling in all of these iterative procedures?
}
To formulate a sufficient conjecture which would guarantee that without-replacement sampling outperforms with-replacement, let us first formalize some notation.  Throughout,
$[n]$ denotes the set of integers from $1$ to $n$. Let $\mathbb{D}$ be some domain, $f:\mathbb{D}^k \rightarrow \R$, and  $(x_1,\ldots,x_n)$ a set of $n$ elements from $\mathbb{D}$.  We define the without-replacement expectation as
\[
	\E_\wo[ f(x_{i_1},\ldots,x_{i_k})] = \tfrac{(n-k)!}{n!}\sum_{j_1\neq j_2 \neq \ldots \neq j_k}f(x_{j_1},\ldots,x_{j_k})\,.
\]
That is, we average the value of $f$ over all ordered tuples of elements from $(x_1,\ldots,x_n)$.   Similarly, the with-replacement expectation is defined as
\[
	\E_\wr[ f(x_{i_1},\ldots,x_{i_k})] = n^{-k} \sum_{(j_1,\ldots,j_k)=1}^n\ f(x_{j_1},\ldots,x_{j_k})\,.
\]

With these conventions, we can list our main conjectures as follows:

\begin{conjecture}[Operator Inequality of Noncommutative Arithmetic and Geometric Means]\label{conj:ncamgm}
	Let $\mtx{A}_1,\ldots,\mtx{A}_n$ be a collection of positive semidefinite matrices. Then we conjecture that the following two inequalities always hold:
	\begin{align}
\label{eq:bias-conj}		\left\| \E_{\wo}\left[ \prod_{j=1}^k \mtx{A}_{i_j} \right]\right\|&\leq
		\left\| \E_{\wr}\left[ \prod_{j=1}^k \mtx{A}_{i_j} \right]\right\|\\
		 \label{eq:weak-variance-conj}		\left\| \E_{\wo}\left[ \prod_{j=1}^k \mtx{A}_{i_{k-j+1}}\prod_{j=1}^k \mtx{A}_{i_j}\right] \right\| 
		 &\leq \left\| \E_{\wr}\left[ \prod_{j=1}^k \mtx{A}_{i_{k-j+1}}\prod_{j=1}^k \mtx{A}_{i_j}\right] \right\| 
		 \end{align}
\end{conjecture}
\noindent Note that in~\eq{bias-conj}, we  have $\E_{\wr}[ \prod_j \mtx{A}_{i_j}] = (\tfrac{1}{n}\sum_i \mtx{A}_i)^k=(\mtx{M}_A)^k$.

Assuming this conjecture holds, let us return to the analysis of the IGM~\eq{main-sgd-exp}. Assuming that $\vct{x}_0-\vct{x}_\star$ is an arbitrary starting vector and that \eq{weak-variance-conj} holds, we have that each term in this summation is smaller for the without-replacement sampling model than for the with-replacement sampling model.  In turn, we expect the without-replacement sampling implementation will return lower risk after one pass over the data-set.  Similarly, for the randomized Kaczmarz iteration~\eq{main-kacz-exp}, Conjecture~\ref{conj:ncamgm} implies that a without-replacement sample will have lower error after $k<n$ iterations.

In the remainder of this document we provide several case studies illustrating that these noncommutative variants of the arithmetic-geometric mean inequality hold in a variety of settings, establishing along the way tools and techniques that may be useful for proving Conjecture~\ref{conj:ncamgm} in full generality.

\subsection{Two matrices and a search for the geometric mean}\label{sec:two-matrices}

Both of the inequalities~\eq{bias-conj} and~\eq{weak-variance-conj} are true when $n=2$.  These inequalities all follow from an well-estabilished line of research in estimating the norms of products of matrices, started by the seminal work of~\citet{Bhatia90}.  

\begin{proposition}\label{prop:two-matrices} Both~\eq{bias-conj} and~\eq{weak-variance-conj} hold when $n=2$
\end{proposition}
\begin{proof}  Let $\mtx{A}$ and $\mtx{B}$ be positive definite matrices.  Both of our arithmetic-geometric mean inequalities follow from the stronger inequality
\begin{equation}\label{eq:bhatia1}
	\left\| \mtx{A}\mtx{B} \right\| \leq  \left\| \tfrac{1}{2}\mtx{A} +\tfrac{1}{2} \mtx{B} \right\|^2\,.
\end{equation}	
This bound was proven by~\cite{Bhatia00}.  In particular, since
\[
	\|\tfrac{1}{2}\mtx{A}\mtx{B} + \tfrac{1}{2}\mtx{B}\mtx{A}\| \leq \left\| \mtx{A}\mtx{B} \right\|\,,
\]
\eq{bias-conj} is immediate.  For~\eq{weak-variance-conj}, note that
\begin{equation}\label{eq:two-mat-wo}
	\E_{\wo}[ \mtx{A}_{i_1}\mtx{A}_{i_2}^2 \mtx{A}_{i_1}] =  \tfrac{1}{2}\mtx{A}\mtx{B}^2\mtx{A}+\tfrac{1}{2}\mtx{B}\mtx{A}^2\mtx{B}
\end{equation}
and
\[
	\E_{\wr}[ \mtx{A}_{i_1}\mtx{A}_{i_2}^2 \mtx{A}_{i_1}] =  \tfrac{1}{4} \mtx{A}^4 + \tfrac{1}{4}\mtx{A}\mtx{B}^2\mtx{A}+\tfrac{1}{4}\mtx{B}\mtx{A}^2\mtx{B}+\tfrac{1}{4}\mtx{B}^4\,.
\]

We can bound~\eq{two-mat-wo} by
\begin{align*}
\tfrac{1}{2}\|\mtx{A}\mtx{B}^2\mtx{A} + \mtx{B}\mtx{A}^2\mtx{B}\| \leq	\|\mtx{A}\mtx{B}^2\mtx{A}\|
=\|\mtx{A}\mtx{B}\|^2 \leq \| \tfrac{1}{2}\mtx{A}+\tfrac{1}{2}\mtx{B}\|^4= \| (\tfrac{1}{2}\mtx{A}+\tfrac{1}{2}\mtx{B})^4\|\,.
\end{align*}
Here, the first inequality is the triangle inequality and the subsequent equality follows because the norm of $\mtx{X}^T\mtx{X}$ is equal to the squared norm of $\mtx{X}$.  The second inequality is~\eq{bhatia1}.

To complete the proof we show 
\[
\mtx{X}_L:=(\tfrac{1}{2}\mtx{A}+\tfrac{1}{2}\mtx{B})^4 \preceq \tfrac{1}{4} \mtx{A}^4 + \tfrac{1}{4}\mtx{A}\mtx{B}^2\mtx{A}+\tfrac{1}{4}\mtx{B}\mtx{A}^2\mtx{B}+\tfrac{1}{4}\mtx{B}^4:=\mtx{X}_R
\]
in the semidefinite ordering.  But this follows by observing
\begin{align}\label{eq:ncsos}
	\mtx{X}_R-\mtx{X}_L 
= \sum_{ p\in [2]^2}\sum_{q\in [2]^2 } Q(p,q) \mtx{A}_{p(2)}\mtx{A}_{p(1)} \mtx{A}_{q(1)} \mtx{A}_{q(2)}
\end{align}
where $Q(p,q) = 3/16$ if $p=q$ and $-1/16$ otherwise.  Since $p$ and $q$ both take $4$ possible values, the matrix $Q$ is positive definite which means that $\mtx{X}_R-\mtx{X}_L$ can be written as a nonnegative sum of products of the form $\mtx{Y}\mtx{Y}^T$.  We conclude that $\mtx{X}_R-\mtx{X}_L$ must be positive define and hence $\|\mtx{X}_R\| \geq \|\mtx{X}_L\|$, completing the proof\footnote{An explicit decomposition~\eq{ncsos} into Hermitian squares was initially found using the software NCSOSTools by~\citet{NCSOSTools}.  This software finds decompositions of matrix polynomials into sums of Hermitian squares.  Our argument was constructed after discovering this decomposition.}.
\end{proof}

Note that this proposition actually verifies a stronger statement: for two matrices, the arithmetic-geometric mean inequality holds for \emph{deterministic} orderings of two matrices.  We will discuss below how symmetrization is necessary for more than two matrices.  In fact, considerably stronger inequalities hold for symmetrized products of two matrices. As a striking example, the symmetrized geometric mean actually precedes the square of the arithmetic mean in the positive definite order.  Let $\mtx{A}$ and $\mtx{B}$ be positive semidefinite.  Then we have
\[
 \left(\tfrac{1}{2} \mtx{A} + \tfrac{1}{2}\mtx{B}\right)^2 - \left(\tfrac{1}{2} \mtx{A} \mtx{B} + \tfrac{1}{2} \mtx{B} \mtx{A}\right) =  \tfrac{1}{4} \mtx{A}^2 + \tfrac{1}{4}\mtx{B}^2 - \tfrac{1}{4} \mtx{A} \mtx{B} -\tfrac{1}{4} \mtx{B} \mtx{A}=  \left(\tfrac{1}{2} \mtx{A} - \tfrac{1}{2}\mtx{B}\right)^2\succeq 0\,.
\]
This ordering breaks for $3$ matrices as evinced by the counterexample
\[
	\mtx{A}_1 = \left[\begin{array}{cc} 7 & 0\\ 0 & 0 \end{array}\right]\,,~~~
	\mtx{A}_2 = \left[\begin{array}{cc} 1 & 1\\ 1 & 1 \end{array}\right]\,,~~~
	\mtx{A}_3 = \left[\begin{array}{cc} 1 & 1\\ 1 & 1 \end{array}\right]\,.
\]
The interested reader should consult~\citet{Bhatia08} for a comprehensive list of  inequalities concerning pairs of positive semidefinite matrices.

Unfortunately, these techniques are specialized to the case of two matrices, and no proof currently exists for the inequalities when $n\geq 3$.   There have been a varied set of attempts to extend the noncommutative arithmetic-geometric mean inequalities to more than two matrices. Much of the work in this space has focused on how to properly define the geometric mean of a collection of positive semidefinite matrices.  For instance,~\citet{Ando04} demarcate a list of properties desirable by any geometric mean, with one of the properties being that the geometric mean must precede the arithmetic mean in the positive-definite ordering.  Ando~\emph{et al} derive a geometric mean satisfying all of these properties, but the resulting mean in no way resembles the means of matrices discussed in this paper.  Instead, their geometric mean is defined as a fixed point of a nonlinear map on matrix tuples.  ~\citet{Bhatia06} and~\citet{Bonnabel09} propose geometric means based on geodesic flows on the Riemannian manifold of positive definite matrices, however these means also do not correspond to the averaged matrix products that we study in this paper.

\subsection{When is it not necessary to symmetrize the order?}\label{sec:commutative}

When the matrices commute, Conjecture~\ref{conj:ncamgm} is a consequence of the standard arithmetic-geometric mean inequality (more precisely, a consequence of Maclaurin's inequalities).   

\begin{theorem}[Maclaurin's Inequalities]
	Let $x_1,\ldots, x_n$ be positive scalars.  Let 
	\[
	s_k = {n \choose k}^{-1} \sum_{\substack{\Omega \subset [n] \\|\Omega|=k}}\prod_{i\in \Omega} x_i
	\]
be the normalized $k$th symmetric sum.  Then we have
\[
	s_1 \geq \sqrt{s_2} \geq \ldots \geq \sqrt[n-1]{s_{n-1}} \geq \sqrt[n]{s_{n}}
\]
\end{theorem}

Note that $s_1 \geq \sqrt[n]{s_n}$ is the standard form of the arithmetic-geometric mean inequality.  See~\citet{HardyLittlewoodPolyaBook} for a discussion and proof of this chain of inequalities.

To see that these inequalities immediately imply Conjecture~\ref{conj:ncamgm} when the matrices $\mtx{A}_i$ are mutually commutative, note first that when $d=1$, we have
\[
\E_{\wo}\left[ \prod_{i=1}^k a_i\right] =  {n \choose k}^{-1} \sum_{\substack{\Omega \subset [n] \\|\Omega|=k}}\prod_{i\in \Omega} a_i \leq \left(\frac{1}{n}\sum_{i=1}^n a_i\right)^k = \E_{\wr}\left[ \prod_{i=1}^k a_i\right] \,.
\]
The higher dimensional analogs follow similarly.  If all of the $\mtx{A}_i$ commute, then the matrices are mutually diagonalizable.  That is, we can write $\mtx{A}_i = \mtx{U} \Lambda_i \mtx{U}^T$ where $\mtx{U}$ is an orthogonal matrix, and the $\Lambda_i=\diag(\lambda_1^{(i)},\ldots,\lambda_d^{(i)})$ are all diagonal matrices of the eigenvalues of $\mtx{A}_i$ in descending order.  Then we have
\[
\left\| \E_{\wo}\left[ \prod_{i=1}^k \mtx{A}_i \right]\right\|=
\left\| \E_{\wo}\left[ \prod_{i=1}^k \mtx{\Lambda}_i \right]\right\|= \E_{\wo} \prod_{i=1}^k \lambda_1^{(i)} \leq \E_{\wr} \prod_{i=1}^k \lambda_1^{(i)}  = \left\| \E_{\wr}\left[ \prod_{i=1}^k \mtx{A}_i \right]\right\|
\]
and we also have
\[
\left\| \E_{\wo}\left[ \prod_{i=1}^k \mtx{A}_i \right]\right\|= \E_{\wo}\left[ \left\|\prod_{i=1}^k \mtx{A}_i \right\|\right]~~~\mbox{and}~~~\left\| \E_{\wr}\left[ \prod_{i=1}^k \mtx{A}_i \right]\right\|= \E_{\wr}\left[ \left\|\prod_{i=1}^k \mtx{A}_i \right\|\right]
\]
verifying our conjecture.  In fact, in this case, any order of the matrix products will satisfy the desired arithmetic-geometric mean inequalities.

\subsection{When is it necessary to symmetrize the order?}

In contrast, symmetrizing over the order of the product is necessary for noncommutative operators.  The following example, communicated to us by Aram~\citet{AramGossip}, provides deterministic without-replacement orderings that have exponentially larger norm than the with-replacement expectation.  Let $\omega_n = \pi/n$. For $n\geq 3$, define the collection of vectors
\begin{equation}\label{eq:harmonic-frames-def}
	\vct{a}_{k;n} = \left[\begin{array}{c} \cos\left( k \omega_n \right) \\ \sin \left(k \omega_n \right)\end{array}\right]\,.
\end{equation}
Note that all of the $\vct{a}_{k;n}$ have norm $1$ and, for $1\leq k <n$,  $\langle \vct{a}_{k;n}, \vct{a}_{k+1;n}\rangle = \cos\left(\omega_n \right)$.  The matrices $\mtx{A}_k := \vct{a}_{k;n} \vct{a}_{k;n}^T$ are all positive semidefinite for $1\leq k\leq n$, and we have the identity
\begin{equation}\label{eq:frame-def}
	\frac{1}{n} \sum_{k=1}^n \mtx{A}_k = \tfrac{1}{2} \mtx{I}\,.
\end{equation}
Any set of unit vectors satisfying~\eq{frame-def} is called a \emph{normalized tight frame}, and the vectors~\eq{harmonic-frames-def} form a \emph{harmonic frame} due to their trigonometric origin~\citep[see][]{Hassibi01,Goyal01}. The product of the $\mtx{A}_i$ is given by
\[
	\prod_{i=1}^k \mtx{A}_i = \vct{a}_{k;n} \vct{a}_{1;n}^T \prod_{j=1}^{k-1} \langle \vct{a}_{j;n}, \vct{a}_{j+1;n}\rangle  = \vct{a}_{k;n} \vct{a}_{1;n}^T\cos^{k-1}\left(\omega_n\right)\,,
\]
and hence
\begin{equation*}\label{eq:gm}
	\left\|\prod_{i=1}^k \mtx{A}_i\right\| = \cos^{k-1}\left( \omega_n \right) \geq  2^k\cos^{k-1}(\omega_n) \left\|\left(\frac{1}{n}\sum_{k=1}^n \mtx{A}_k \right)^k\right\| \,.
\end{equation*}
Therefore, the arithmetic mean is less than the deterministically ordered matrix product for all $n\geq 3$. 

It turns out that this harmonic frame example is in some sense the worst case.  The following proposition shows that the geometric mean is always within a factor of $d^k$ of the arithmetic mean for any ordering of the without-replacement matrix product.

\begin{proposition}\label{prop:worst-case}
Let $\mtx{A}_1,\ldots,\mtx{A}_n$ be $d\times d$ positive semidefinite matrices.  Then
\[
	\left\| \E_{\wo}\left[ \prod_{i=1}^k \mtx{A}_{j_i} \right] \right\| \ \leq  d^k \left\| \E_{\wr} \left[ \prod_{i=1}^k \mtx{A}_{j_i} \right] \right\|
\]
\end{proposition}

\begin{proof}
If we sample $j_1,\ldots,j_k$ uniformly from $[n]$, then we have
\[
\begin{aligned}
\left\| \E_{\wo}\left[ \prod_{i=1}^k \mtx{A}_{j_i} \right] \right\|  &\leq \E_{\wo}\left[\left\| \prod_{i=1}^k \mtx{A}_{j_i} \right\|\right] \leq  \E_{\wo}\left[\prod_{i=1}^k \left\| \mtx{A}_{j_i} \right\|\right]
\leq  \E_{\wo}\left[\prod_{i=1}^k  \trace(\mtx{A}_{j_i})  \right]\\
&\leq \left(\frac{1}{n} \sum_{i=1}^n  \trace(\mtx{A}_i) \right)^k= \trace\left(\frac{1}{n} \sum_{i=1}^n  \mtx{A}_i \right)^k \leq \left\|\frac{d}{n} \sum_{i=1}^n  \mtx{A}_i \right\|^k
\end{aligned}
\]
Here, the first inequality follows from the triangle inequality.  The second, because the operator norm is submultiplicative.  The third inequality follows because the trace dominates the operator norm.  The fourth inequality is Maclaurin's.  The fifth inequality follows because the trace of a $d\times d$ positive semidefinite matrix is upper bounded by $d$ times the operator norm.  The final inequality is again the triangle inequality.
\end{proof}

Note that this worst-case bound holds for deterministic orders of matrix products as well. Once we apply the submultiplicative property of the operator norm, all of the non-commutativity is washed out of the problem.  Examples of deterministic matrix products saturating this upper bound can be constructed in higher dimensions using frames.  If $d$ even, set 
\begin{equation}\label{eq:even-d-frames}
\vct{f}_{k+1}^T = \sqrt{\frac{2}{d}} \left[\vct{a}_{k;n}^T, \vct{a}_{3k;n}^T, \cdots, \vct{a}_{(d-1)k;n}^T\right]~~~\mbox{for}~k=0,1,\ldots, {n-1}\,,
\end{equation}
and for odd $d$
\begin{equation}\label{eq:odd-d-frames}
\vct{f}_{k+1}^T = \sqrt{\tfrac{2}{d}} \left[  \tfrac{1}{\sqrt{2}}, \vct{a}_{2k;n}^T
 \vct{a}_{4k;n}^T, \cdots, \vct{a}_{(d-1) k;n}^T \right]~~~\mbox{for}~k=0,1,\ldots, {n-1}\,.
\end{equation}
Then one can verify again using standard trigonometric identities that
\[
	\frac{1}{n}\sum_{k=1}^n \vct{f}_k \vct{f}_k^T = \frac{1}{d}\mtx{I}\,,
\]
and that the inner products of adjacent $\vct{f}_i$ are
\[
\begin{aligned}
\vct{f}_{i}^T\vct{f}_{i+1} &= \begin{cases} \frac{2}{d} \cos\left((d/2-1) \omega_n \right)\frac{\sin\left((d/2+1) \omega_n \right)}{\sin\left(\omega_n\right)} - \frac{2}{d}\cos\left( \omega_n \right)  & d~\text{even}
\\
 \frac{2}{d} \cos\left( (d-1)/2 \omega_n \right)\frac{\sin\left((d+1)/2 \omega_n \right)}{\sin\left(\omega_n\right)}-\frac{1}{d} & d~\text{odd}
\end{cases}\,.
\end{aligned}
\]
These inner products are approximately $1- \frac{\pi^2(d^2-1)}{6n^2}$ for large $n$.  Thus, each of these cases violate the arithmetic-geometric mean inequality for the order $(1,2,\ldots, k)$ by a factor of approximately $d^k$ provided $n\geq d$. 

At first glance, the harmonic frames example appears to cast doubt on the validity of Conjecture~\ref{conj:ncamgm}.  However, after symmetrizing over the symmetric group, we can show that the $d=2$ harmonic frames do obey~\eq{bias-conj}.

\begin{theorem}\label{thm:harmonic-frames} Let $\lambda(n)= \,_2F_3\left[\begin{array}{ccc}
1 & -n/2+1/2 & -n/2\\
1/2 & -n+1 
  \end{array}; 1 \right]$. With the $\vct{a}_{k;n}$ defined in~\eq{harmonic-frames-def},
\[ \tfrac{1}{n!}\sum_{\sigma \in S_n} \prod_{i=1}^n \vct{a}_{\sigma(i);n} \vct{a}_{\sigma(i+1);n}^T = - \lambda(n) 2^{-n} \mtx{I}\,,
~~~\text{and}~~~
1\geq \lambda(n) =\mathcal{O}(n^{-1})\,.
 \]
 \end{theorem}

\noindent
This theorem additionally verifies that there is an asymptotic gap between the arithmetic and geometric means of the harmonic frames example after symmetrization.  We include a full proof of this result in Appendix~\ref{sec:frames}.   The proof treats the norm variationally using the identity that $\|\mtx{X}\|_2$ is the maximum of  $\mtx{v}^{T}\mtx{X}\vct{v}$ over all unit vectors $\vct{v}$. Our computation then reduces to effectively computing a Fourier transform of the function of $\vct{v}$ in an appropriately defined finite group.  We show that the Fourier coefficients can be viewed as enumerating sets, and we compute them exactly using generating functions.

The combinatorial argument that we use to prove Theorem~\ref{thm:harmonic-frames} is very specialized.  To provide a broader set of examples, we now turn to show that Conjecture~\ref{conj:ncamgm} does in fact hold for many classes of \emph{random} matrices.

\section{Random matrices}\label{sec:random}

In this section, we show that if $\mtx{A}_1,\ldots, \mtx{A}_n$ are generated i.i.d. from certain distributions, then Conjecture~\ref{conj:ncamgm} holds in expectation with respect to the $\mtx{A}_i$.  Section~\ref{sec:random-ncamgm} assumes that $\mtx{A}_i = \mtx{Z}_i \mtx{Z}_i^T$ where $\mtx{Z}_i$ have independent entries, identically sampled from some symmetric distribution.  In Section~\ref{sec:sgd-consequences}, we explore when the matrices $\mtx{A}_i$ are random rank-one perturbations of the identity as was the case in the IGM and Kaczmarz examples. 

\subsection{Random matrices satisfy the noncommutative arithmetic-geometric mean inequality}\label{sec:random-ncamgm}

In this section, we prove the following
\begin{proposition}\label{thm:random}
For each $i=1,\ldots,n$, suppose $\mtx{A}_i=\vct{Z}_i \vct{Z}_i^T$ with $\vct{Z}_i$ a $d\times r$ random matrix whose entries are i.i.d. samples from some symmetric distribution.  Then Conjecture~\ref{conj:ncamgm} holds in expectation.
\end{proposition}

\begin{proof}
Suppose the entries of each $\mtx{Z}_i$ have finite variance $\sigma^2$ (the theorem would be otherwise vacuous if we assumed infinite variance).  Let the $(a,b)$ entry of $\mtx{Z}_i$ be denoted by $Z_{a,b}^{(i)}$.  Also, denote by $\mtx{W}$ the matrix with all of the $\mtx{Z}_i$ stacked as columns: $\mtx{W} = \sigma^{-1} [\mtx{Z}_1,\ldots,\mtx{Z}_n]$.

Let's first prove that \eq{bias-conj} holds in expectation for these matrices.  First, consider the without-replacement samples, which are considerably easy to analyze. Let $(j_1,\ldots,j_k)$ be a without-replacement sample from $[n]$.  Then
\[
	\left\|\E\left[ \prod_{i=1}^k \mtx{A}_{j_i} \right] \right\|= \|\E[\mtx{A}_1]^k\| =  r^k\sigma^{2k}\,.
\]


For the arithmetic mean, we can compute
\begin{align}
\nn		&r^{-k} \sigma^{-2k}\left\|\E\left[ \left( \frac{1}{n} \sum_{i=1}^n \mtx{A}_i \right)^k \right]\right\|
 \geq	\frac{	1}{r^k \sigma^{2k} d} \trace\left(\E\left[ \left( \frac{1}{n} \sum_{i=1}^n \mtx{A}_i \right)^k \right]\right)\\
\label{eq:wishart-power}=& \E\left[ d^{-1}\trace\left( \left( \frac{1}{nr \sigma^2} \sum_{i=1}^n \mtx{A}_i \right)^k \right)\right]=  \E\left[ d^{-1}\trace\left( \left( \frac{1}{nr}\mtx{W}\mtx{W}^T \right)^k \right)\right]\\
\nn =& d^{-1}(nr)^{-k}  \sum_{\{a_1,\ldots,a_k\}=1}^d\sum_{\{b_1,\ldots,b_k\}=1}^{nr}    \E[W_{a_1,b_1} W_{a_2,b_1} W_{a_2,b_2} W_{a_3,b_2} \ldots W_{a_k , b_k}W_{a_1,b_k} ]\\
\label{eq:nasty1} =& (nr)^{-k}\sum_{\{a_2,\ldots,a_k\}=1}^d\sum_{\{b_1,\ldots,b_k\}=1}^{nr}    \E[W_{1,b_1} W_{a_2,b_1} W_{a_2,b_2} W_{a_3,b_2} \ldots W_{a_k , b_k}W_{1,b_k} ]\,.
	\end{align}
Note that since $W_{ij}$ are iid, symmetric random variables, each term in this sum is zero if it contains an odd power of $W_{ij}$ for some $i$ and $j$.  If all of the powers in a summand are even, its expected value is bounded below by $1$.  A simple lower bound for this final term~\eq{nasty1} thus looks only at the contribution from when all of the indices $a_i$ are set equal to $1$.
\[
(nr)^{-k}\sum_{\{b_1,\ldots,b_k\}=1}^{nr}    \E[W_{1,b_1}^2 W_{1,b_2}^2 \ldots W_{1 , b_k}^2 ] = \E\left[ \left( \frac{1}{nr} \sum_{b=1}^{nr} W_{1,b}^2\right)^k\right]  \geq
\left( \E\left[  \frac{1}{nr} \sum_{b=1}^{nr} W_{1,b}^2 \right] \right)^k  = 1\,.
\]
Here the inequality is Jensen's. This calculation proves~\eq{bias-conj} for our family of random matrices.  That is, we have demonstrated that the expected value of the with-replacement sample has greater norm than the expected value of the without-replacement sample.

To verify that~\eq{weak-variance-conj} holds for our random matrix model, we first record the following property about the fourth moments of the entries of the $\mtx{A}_i$.  Let $\xi:= \E[G_{ij}^4]^{1/4}$.   Then we can verify by direct calculation that
\begin{equation}\label{eq:mom-calc}
 \E[A^{u}_{i_1,j_1}A^{u}_{i_2,j_2}] = 
\begin{cases} 
r(r-1) \sigma^{4} + r\xi^4 & \{i_1,j_2\} = \{i_2,j_2\} \text{ and } i_1 = i_2 \\
r \sigma^{4} & \{i_1,j_2\} = \{i_2,j_2\} \text{ and } i_1 \neq i_2 \\
0 & \text{otherwise}
\end{cases}
\end{equation}

A consequence of this lemma is that $\E[ \mtx{A}_i^2 ] = (r(r+d-1)\sigma^4 + r\xi^4)\mtx{I}_d$.  Using this identity, we can set $\zeta:= r(r+d-1)\sigma^4 + r\xi^4$ and we then have
\begin{align*}
\E[\mtx{V}(i_1,\dots,i_k)] 
= \E[ \mtx{A}_{i_k} \dots \mtx{A}_{i_1}^2 \dots \mtx{A}_{i_k}]
= \zeta \E[ \mtx{A}_{i_k} \dots \mtx{A}_{i_2}^2 \dots \mtx{A}_{i_k} ]=
\cdots=  \zeta^{k}  \mtx{I}_d
\end{align*}

We compute this identity in a second way that describes its
combinatorics more explicitly, which we will use as to derive our
lower bound.
\begin{align*}
\E[V_{u,v}( i_1,\dots,i_k)] = & 
\E\left[\sum_{ p \in [d]^{2k}}   A^{(i_1)}_{u,p_1}A^{(i_{2n})}_{p_2,v} \prod_{j=2}^{k-1} A^{(i_j)}_{p_{i_{j-1}},p_{i_j}} A^{(i_{2k - j + 1})}_{p_{i_{2k - j}},p_{i_{2k - j + 1}}} \right] \\
= &
\sum_{p \in [d]^{2k}} \E[A^{(i_1)}_{u,p_1}A^{(i_{2n})}_{p_2,v}] \prod_{j=2}^{k-1} \E[A^{(i_j)}_{p_{i_{j-1}},p_{i_j}} A^{(i_{2k - j + 1})}_{p_{i_{2k - j}},p_{i_{2k - j + 1}}}] 
\end{align*}
The second equality uses linearity coupled with the fact that
$i_1,\dots,i_k$ are distinct, hence $\E[A^{(i_j)}_{u,v} A^{(i_l)}_{u',v'}]
= \E[A^{(i_j)}_{u,v}] \E[ A^{(i_l)}_{u',v'}]$ since
elements from distinct matrices are independent. Many of the terms in this sum contain odd powers which are zero.  Using the fact that $\mtx{A} = \mtx{A}^T$, we see
that all terms that are non-zero must contain only products of two forms: $A_{uu}^2$ or $A_{uv}^2$. Then, we can write the sum:
\[ V_{u,v}( i_1,\dots,i_k)= \sum_{p \in [d]^{k}} \E[(A^{(i_1)}_{u,p_{1}})^2]  \prod_{j=2}^{k} \E[(A^{(i_j)}_{p_{j-1},p_{j}})^2] \]

Now consider the case that some index may be repeated (i.e., there
exist $k,l$ such that $i_j = i_l$ for $j \neq l$).  The key
observation is the following. Let $w$ be a real-valued random variable
with a finite second moment. Then,
\begin{equation}
 \E[w^{2p}] \geq \E[w^2]^p \text{ for } p=0,1,\dots,n 
\label{eq:moments}
\end{equation}
With equality only for $p=0,1$.  This is Jensen's inequality applied
to $x^{p}$ for $x \geq 0$ (since $w$ is real then $w^2$ is positive,
and $x^{p}$ is convex on $[0,\infty)$ for $p=0,1,2,\dots$).  To verify the inequality, let $n_i$ be the number of times index $i$ is repeated and observe
\begin{align}
\nn \E[V_{u,v}(i_1,\dots,i_k)] &=
\E[\sum_{\bar p \in [d]^{2}} A^{(i_1)}_{u,p(1)}A^{(i_{2N})}_{p(2),v} \prod_{j=2}^{k-1} A^{(i_j)}_{p(i_{j-1}),p(i_j)} A^{(i_{2k - j + 1})}_{p(i_{2k - j}),p(i_{2k - j + 1})}] \\
\label{eq:drop-terms} &\geq 
\sum_{p \in [d]^{k}} \E[ (A^{(i_1)}_{u,p(1)})^2  \prod_{j=2}^{k}(A^{(i_j)}_{p(j-1),p(j)})^2]\\
\label{eq:jensen-bound}&\geq 
\sum_{p \in [d]^{k}} \E[ (A^{(i_1)}_{u,p(1)})^2]  \prod_{j=2}^{k} \E[(A^{(i_j)}_{p(j-1),p(j)})^2]\,.
\end{align}
\eq{drop-terms} follows from~\eq{mom-calc}, since all
terms are non-negative.~\eq{jensen-bound} inequality is repeated application
of~\eq{moments}.  The final expression is precisely equal to the without-replacement average.  Now, the with replacement average can be bounded as
\[
\left\|\E_{\wr}\left[ \prod_{j=1}^k \mtx{A}_{i_{k-j+1}}\prod_{j=1}^k \mtx{A}_{i_j}\right] \right\| 
\geq \frac{1}{d} \E_{\wr}\left[ \trace\left(\prod_{j=1}^k \mtx{A}_{i_{k-j+1}}\prod_{j=1}^k \mtx{A}_{i_j}\right] \right)\\
=  \E_{\wr} [V_{1,1}(i_1,\ldots,i_n)]\,.
\]
Since each term in this last expression exceeds the without-replacement expectation, this completes the proof.
\end{proof}

The arguments used to prove Theorem~\ref{thm:random} grossly undercount the number of terms that contribute to the expectation.  Bounds on the quantity~\eq{wishart-power} commonly arise in the theory of random matrices~\citep[see the survey by][for more details and an extensive list of references]{Bai99}.  Indeed, if we let $d=\delta n$ and assume that $W_{ij}$ have bounded fourth moment, we have that~\eq{wishart-power} tends to $(1+\sqrt{\delta})^{2k}$ almost surely a $n \rightarrow \infty$.  That is, the gap between the with- and without-replacement sampling grows exponentially with $k$ in this scaling regime. Similarly, there is an asymptotic, exponential gap between the with and without-replacement expectations in~\eq{weak-variance-conj}.     Observe that~\eq{moments} is strict for $p \geq 2$ for $\chi$-squared random variables.  Thus, for Wishart matrices, if there is even a single repeated value, i.e., $i_j = i_l$ for $j \neq l$, inequality~\eq{jensen-bound} is \emph{strict}.  In Appendix~\ref{sec:gaussian-calc}, we analyze the case where the $\mtx{Z}_i$ are Gaussian (and hence the $\mtx{A}_i$ are Wishart) and demonstrate that the ratio of the expectation is bounded below by $ r e^{\frac{1}{4k(k+1)}}
\left( \frac{16k}{e^2 r (r+d+1)} \right)^{k}$.

\subsection{Random vectors and the incremental gradient method}\label{sec:sgd-consequences}

We can also use a random analysis to demonstrate that for the least-squares problem~\eq{lsopt}, without-replacement sampling outperforms with-replacement sampling if the data is randomly generated. 

Let's look at one step of the recursion~\eq{basic-recursion} and assume that the $\mtx{a}_i$ are sampled i.i.d. from some distribution.  Assume that the moments $\mtx{\Lambda}:=\E[\vct{a}_i \vct{a}_i^T] $ and $\mtx{\Delta}:=\E[\|\vct{a}_i\|^2 \vct{a}_i \vct{a}_i^T] $ exist.  Then we see immediately that
\begin{equation*}\label{eq:wo-sgd-recursion}
\E_{\wo}[\|\vct{x}_k-\vct{x}_\star\|^2]  = \E_{\wo}[\vct{x}_k - \vct{x}_\star]^T (\mtx{I}-2\gamma \mtx{\Lambda} + \gamma^2\mtx{\Delta})\E_{\wo}[\vct{x}_k - \vct{x}_\star] + \rho^2\gamma^2\trace(\mtx{\Lambda})
\end{equation*}
because $\vct{a}_{j_k}$ is chosen independently from $(\vct{a}_{j_{1}}, \ldots,\vct{a}_{j_{k-1}})$
On the other hand, in the with-replacement model, we have
\begin{equation*}\label{eq:wo-sgd-recursion}
\E_{\wr}[\|\vct{x}_k-\vct{x}_\star\|^2]  = \E_{\wr}\left[ (\vct{x}_k - \vct{x}_\star)^T (\mtx{I}-2\gamma \mtx{\Lambda}_n + \gamma^2\mtx{\Delta}_n)(\vct{x}_k - \vct{x}_\star)\right] + \rho^2\gamma^2\trace(\mtx{\Lambda})
\end{equation*}
where
\[
\mtx{\Lambda}_n:= \frac{1}{n} \sum_{i=1}^n \vct{a}_i \vct{a}_i^T ~~~\mbox{and}~~~\mtx{\Delta}_n:=\frac{1}{n} \sum_{i=1}^n \|\vct{a}_i\|^2\vct{a}_i \vct{a}_i^T \,.
\]
In this case, we cannot distribute the expected value because the vector $\vct{x}-\vct{x}_\star$ depends on all $\vct{a}_i$ for $1\leq i \leq n$.  To get a flavor for how these differ, consider the conditional expectation
\[
\begin{aligned}
\E_{\wr}\left[\|\vct{x}_k-\vct{x}_\star\|^2~|~\{\vct{a}_i\}\right] &\leq
 \left(1-2\gamma \lambda_{\mathrm{min}}(\mtx{\Lambda}_n) + \gamma^2\lambda_{\mathrm{max}}(\mtx{\Delta}_n)\right)
\E_{\wr}\left[\|\vct{x}_{k-1}-\vct{x}_\star\|^2~|~\{\vct{a}_i\}\right] \\
&\qquad\qquad\qquad\qquad+ \rho^2\gamma^2\trace(\mtx{\Lambda})
\end{aligned}
\]

Similarly,
\[
	\E_{\wo}\left[\|\vct{x}_k-\vct{x}_\star\|^2\right]  \leq \left(1-2\gamma \lambda_{\mathrm{min}}(\mtx{\Lambda}) + \gamma^2\lambda_{\mathrm{max}}(\mtx{\Delta})\right)
\E_{\wo}\left[\|\vct{x}_{k-1}-\vct{x}_\star\|^2\right] + \rho^2\gamma^2\trace(\mtx{\Lambda})\,.
\]

Expanding out these recursions, we have
\[
	\begin{aligned}
	\E_{\wo}\left[\|\vct{x}_k-\vct{x}_\star\|^2\right]  &\leq \left(1-2\gamma \lambda_{\mathrm{min}}(\mtx{\Lambda}) + \gamma^2\lambda_{\mathrm{max}}(\mtx{\Delta})\right)^k
\E_{\wo}\left[\|\vct{x}_{0}-\vct{x}_\star\|^2\right] + \tfrac{\rho^2\gamma\trace(\mtx{\Lambda})}{2\lambda_{\mathrm{min}}(\mtx{\Lambda})-\gamma\lambda_{\mathrm{max}}(\mtx{\Delta})}\\
	\E_{\wr}\left[\|\vct{x}_k-\vct{x}_\star\|^2~|~\{\vct{a}_i\}\right] &\leq \left(1-2\gamma \lambda_{\mathrm{min}}(\mtx{\Lambda}_n) + \gamma^2\lambda_{\mathrm{max}}(\mtx{\Delta}_n)\right)^k
\E_{\wr}\left[\|\vct{x}_k-\vct{x}_\star\|^2~|~\{\vct{a}_i\}\right]  \\
&\qquad\qquad\qquad\qquad\qquad+ \tfrac{\rho^2\gamma\trace(\mtx{\Lambda}_n)}{2\lambda_{\mathrm{min}}(\mtx{\Lambda}_n)-\gamma\lambda_{\mathrm{max}}(\mtx{\Delta}_n)}
	\end{aligned}
\]

Now, since $\sum_i \vct{a}_i \vct{a}_i^T$ is positive definite and since $\lambda_{\mathrm{min}}$ is concave concave on Hermitian matrices, we have by Jensen's inequalty that
\[
\E[\lambda_{\mathrm{min}}(\mtx{\Lambda}_n)]=\E\left[\lambda_{\mathrm{min}}\left(\frac{1}{n}\sum_{i=1}^n \vct{a}_i \vct{a}_i^T\right) \right] \leq 
\lambda_{\mathrm{min}}\left(\E\left[\frac{1}{n}\sum_{i=1}^n \vct{a}_i \vct{a}_i^T\right]\right)  =
\lambda_{\mathrm{min}}\left( \mtx{\Lambda }\right)\,,
\]
and, since $\lambda_{\mathrm{max}}$ is convex for symmetric matrices,
\[
\E\left[\lambda_{\mathrm{max}}\left(\mtx{\Delta}_n\right)\right]=
\E\left[\lambda_{\mathrm{max}}\left(\frac{1}{n}\sum_{i=1}^n\|\vct{a}_i\|^2 \vct{a}_i \vct{a}_i^T\right)\right]
\geq 	\lambda_{\mathrm{max}}\left(\E\left[\frac{1}{n}\sum_{i=1}^n\|\vct{a}_i\|^2 \vct{a}_i \vct{a}_i^T\right]\right)
=\lambda_{\mathrm{max}}(\mtx{\Delta})\,.
\]
This means that the with-replacement upper bound is worse than the without-replacement estimate with reasonably high probability on most models of $\mtx{a}_i$. Under mild conditions on $\mtx{a}_i$ (including Gaussianity, bounded entries, subgaussian moments, or bounded Orlicz norm), we can estimate tail bounds for the eigenvalues of $\Lambda_n$ and $\Delta_n$~\citep[by applying the techniques of][for example]{Tropp11}.  These large deviation inequalities provide quantitative estimates of the gap between with- and without-replacement sampling for the least mean squares and randomized Kaczmarz algorithms.  Similar, but more tedious analysis, would reveal that with-replacement sampling fares worse with diminishing step sizes as well.

\section{Numerical Evidence}

As described in the introduction, there is substantial numerical evidence that with-replacement sampling underperforms without-replacement sampling in many randomized algorithms.  We invite the interested reader to consult~\citet{Bottou09,RechtRe11,RechtReSIGMOD12}, and many other articles in the machine learning literature to substantiate these empirical claims.  However, for completeness, we provide a few examples demonstrating the gap for the examples in Section~\ref{sec:examples}.

In Figure~\ref{fig:compare}, we display six comparisons of with- and without-replacement sampling.  In the first row, we show the discrepancy when running the randomized Kaczmarz algorithm when the rows of $\mtx{\Phi}$ are the $d$-dimensional, defined by~\eq{even-d-frames}.  In the second row, we plot the results for incremental gradient descent with $\vct{a}_i = \vct{f}_i$ in the same harmonic frames example.  Finally, the third row plots performance when the rows of $\mtx{\Phi}$ are generated i.i.d. from Haar measure on the sphere.  In all three cases, without-replacement sampling converges faster than with-replacement sampling, and when $d$ and $n$ are close, the convergence rate is considerably faster.

\begin{figure}
\centering
\begin{tabular}{cc}
\includegraphics[width=2.5in]{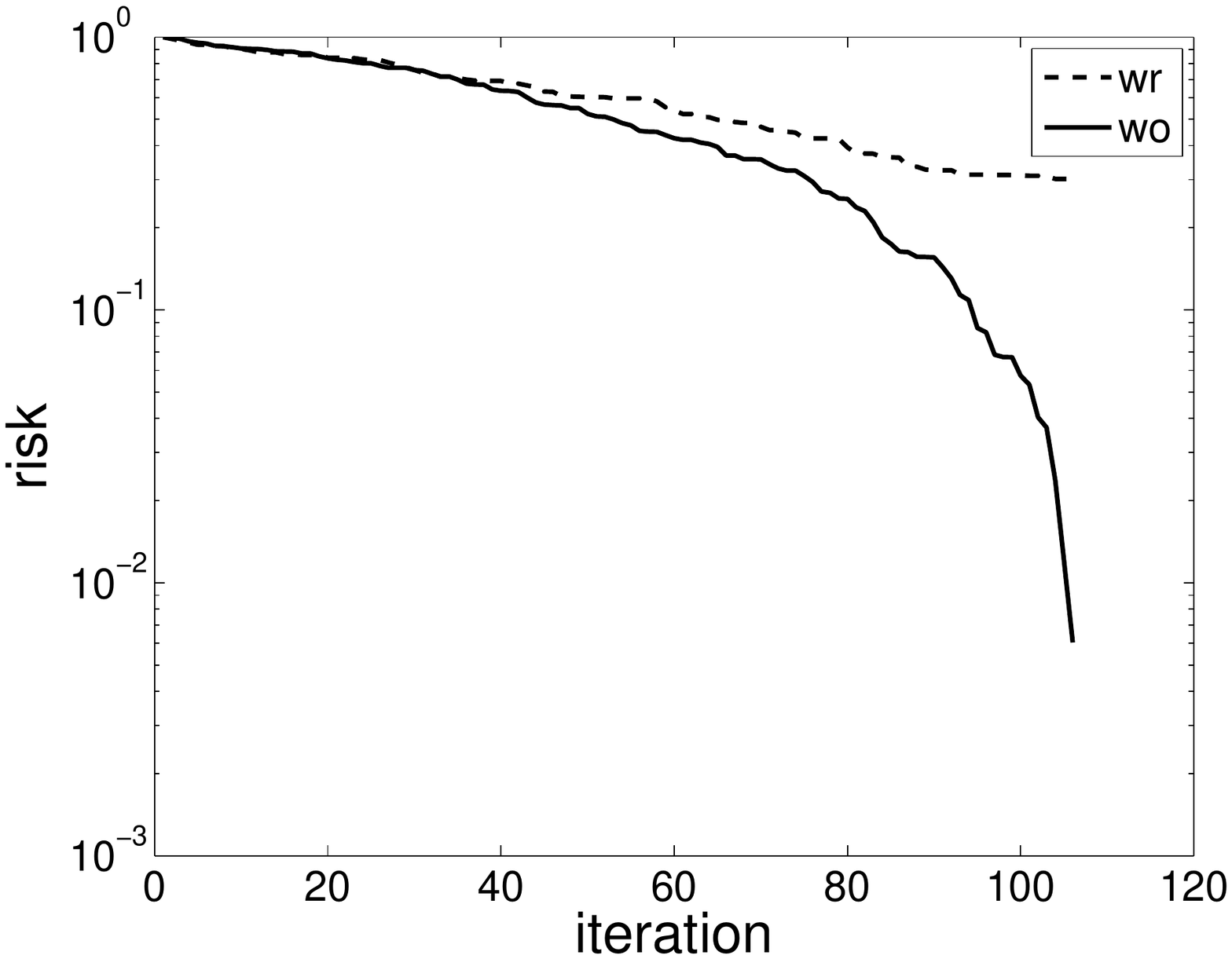}&
\includegraphics[width=2.5in]{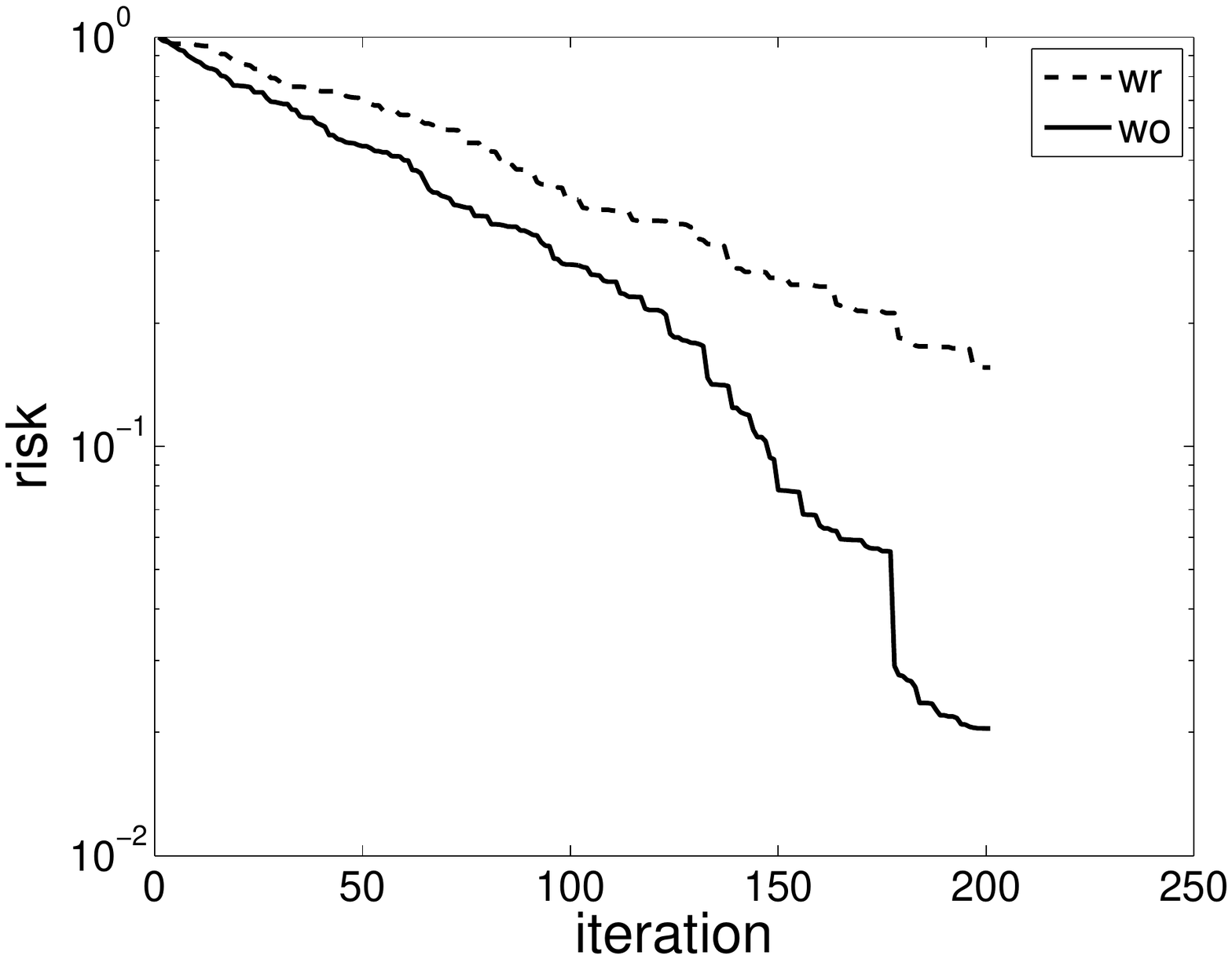}\\
\includegraphics[width=2.5in]{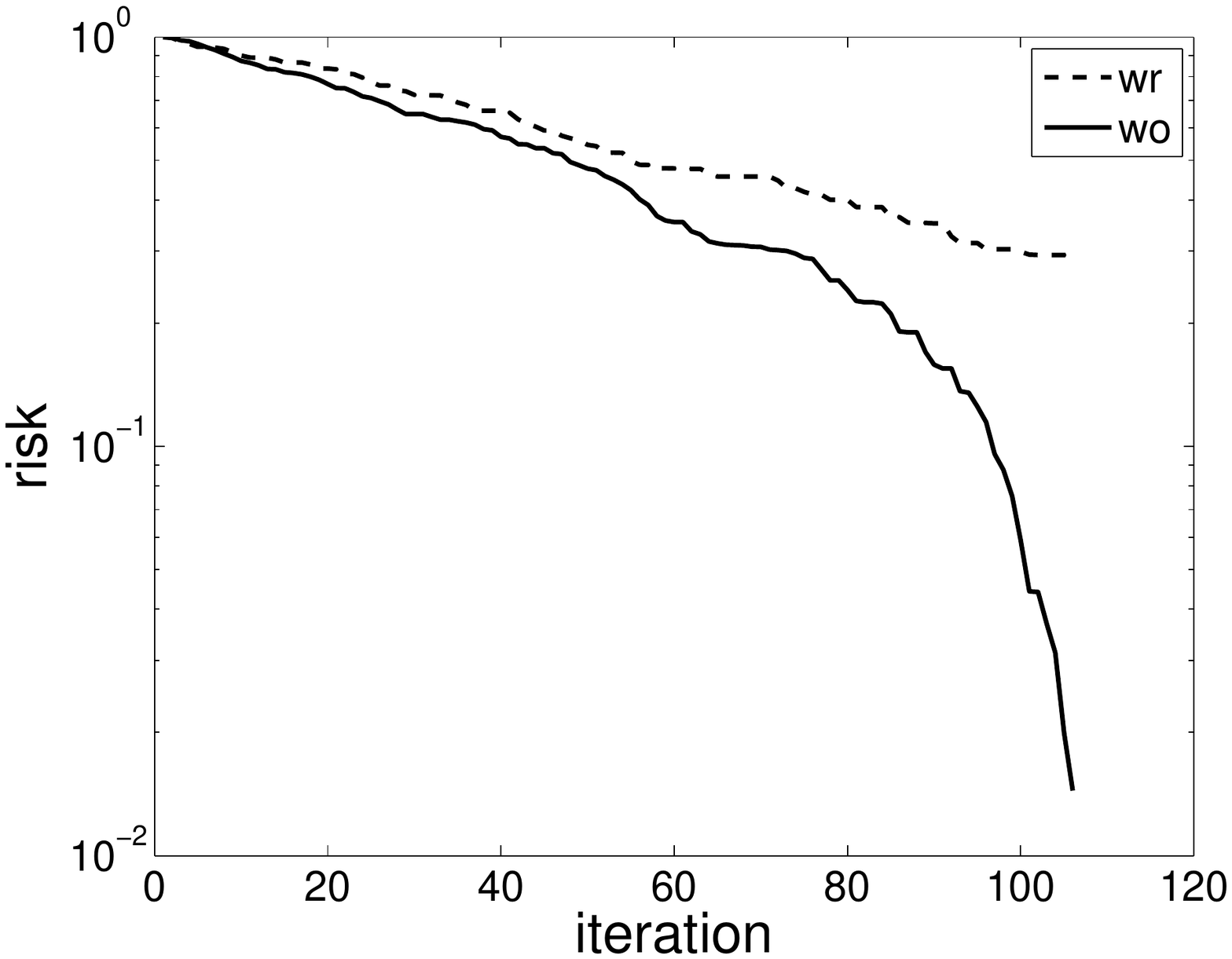}&
\includegraphics[width=2.5in]{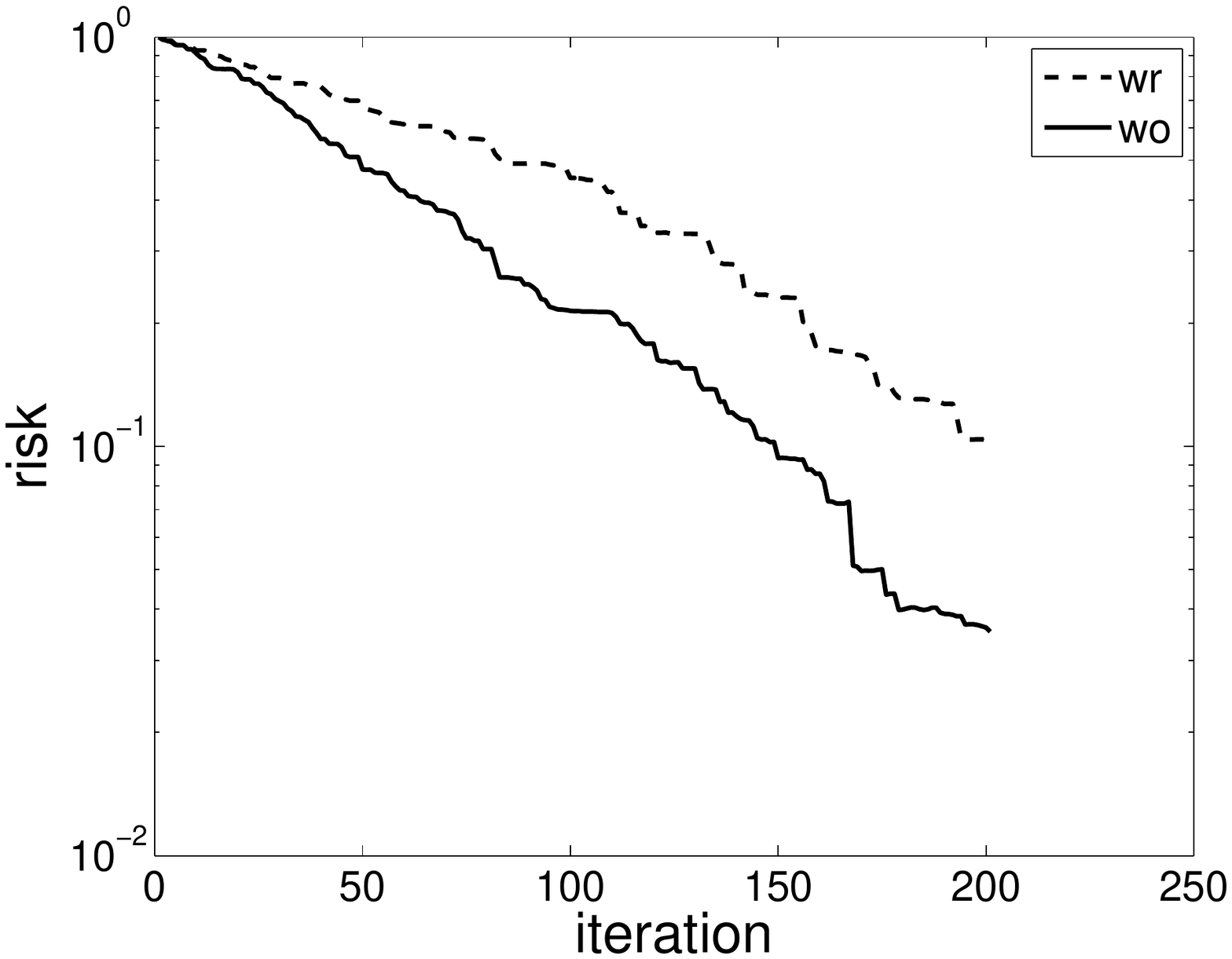}\\
\includegraphics[width=2.5in]{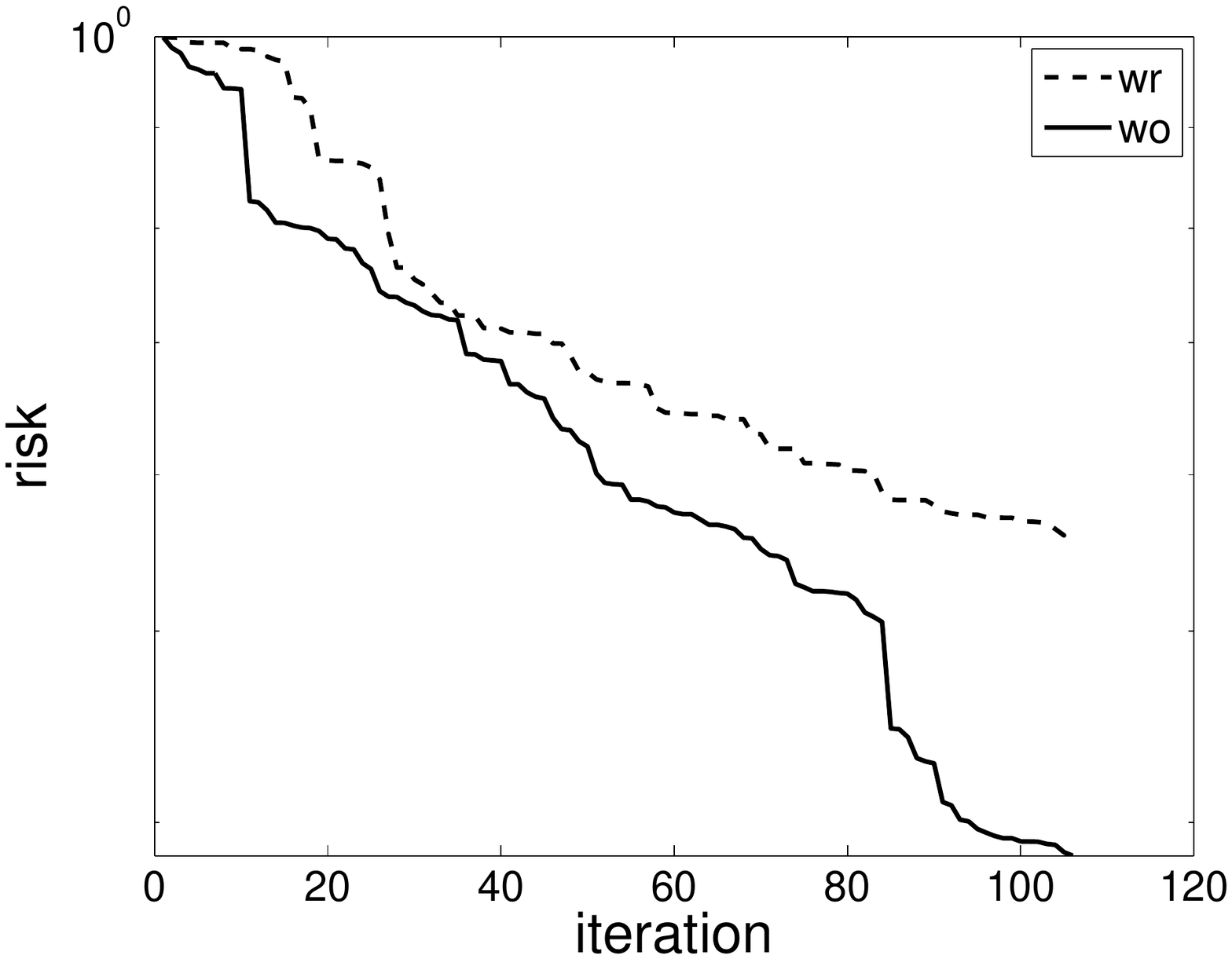}&
\includegraphics[width=2.5in]{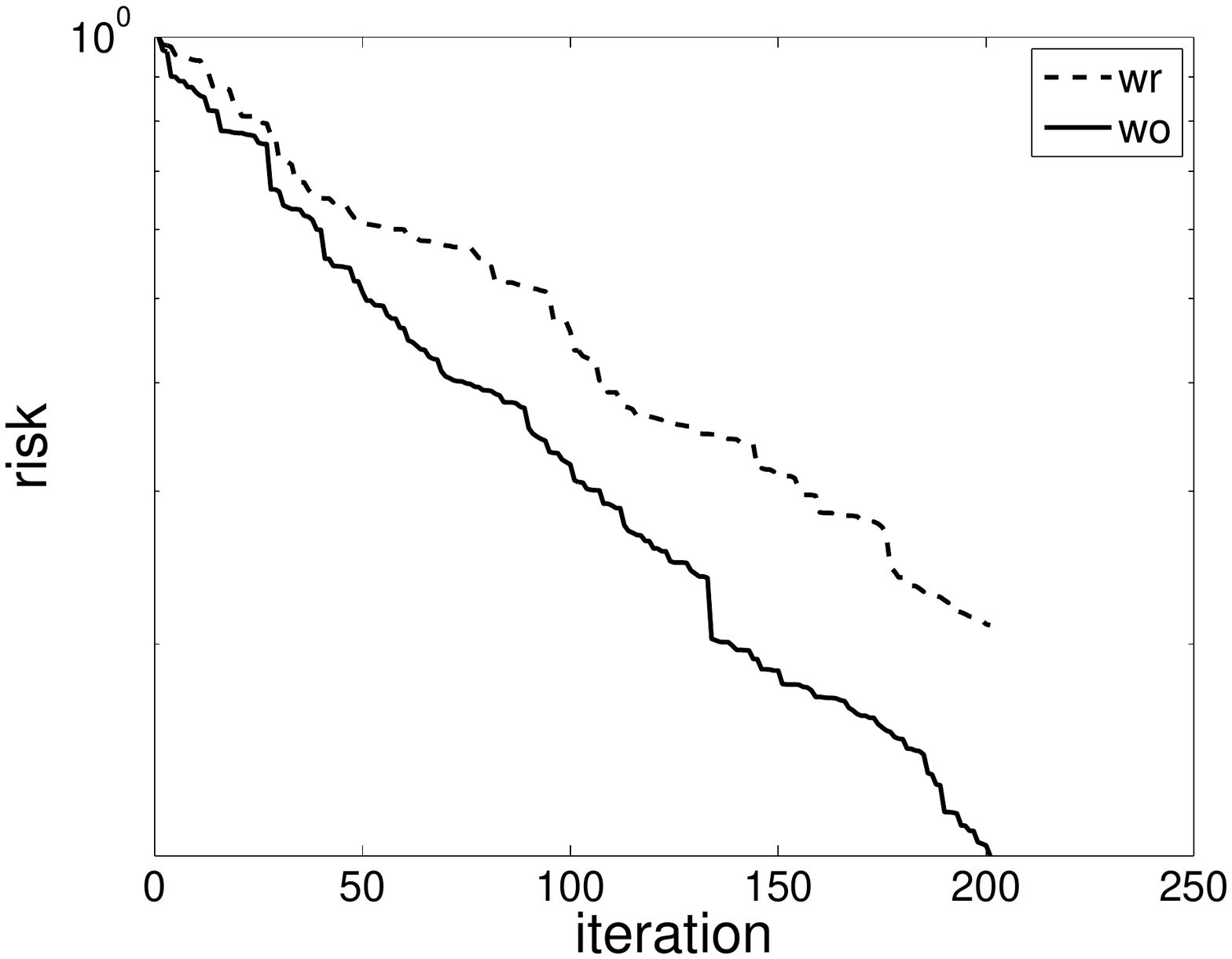}
\end{tabular}
\caption{Comparison of with- and without-replacement sampling on the examples from Section~\ref{sec:examples}.  In the first column, $\mtx{\Phi}$ is $100 \times 105$. In the second column, it is $100 \times 200$.  The first row is the randomized Kaczmarz algorithm with $\mtx{\Phi}$ being the harmonic frame defined in~\eq{even-d-frames}.  The second row is running incremental gradient from Section~\ref{sec:igm} with $\rho=0.01$ where $\vct{a}_i$ is also from the harmonic frame of~\eq{even-d-frames}.  The final row is running Kaczmarz again, this time with vectors generated uniformly from Haar measure on the sphere in $\R^d$.\label{fig:compare}}
\end{figure}

\section{Discussion and open problems}

While i.i.d. matrices are of significant importance in machine learning, the major piece of open work is proving Conjecture~\ref{conj:ncamgm} for all positive semidefinite matrix tuples or finding a counterexample for either of the assertions.  As demonstrated by the harmonic frames example, symmetrized products of deterministic matrices become quickly tedious and difficult to study.  Some sort of combinatorial structure might need to be exploited for a short proof to arise in general.  It remains to be seen if this sort of combinatorics employed in proving Theorem~\ref{thm:harmonic-frames} generalizes beyond this particular example, but we expect these techniques will be useful in future studies of Conjecture~\ref{conj:ncamgm}.  In particular, it would be interesting to see if we could reduce the proof of the conjecture to verifying the conjecture on frames that arise as the orbit of the representation of some finite group.  These frames have been fully classified by~\citet{Hassibi01}, and would reduce Conjecture~\ref{conj:ncamgm} to a finite list of cases.
    
\paragraph{A further conjecture and its consequences} The generalization of~\eq{bhatia1} to $n\geq 3$ asserts a stronger version of~\eq{bias-conj}
\begin{equation}\label{eq:variance-conj}
	\E_{\wo} \left[ \left\| \prod_{j=1}^k \mtx{A}_{i_j}\right\|^2\right] \leq \left\| \tfrac{1}{n} \sum_{i=1}^n \mtx{A}_i \right\|^{2k}\,.
\end{equation}
Certainly, ~\eq{bias-conj} follows from~\eq{variance-conj} by Jensen's inequality the triangle inequality.  Moreover, using the same argument we used in proving Proposition~\ref{prop:two-matrices},~\eq{weak-variance-conj} also follows from~\eq{variance-conj}.  When $n \geq 3$, is it the case that~\eq{variance-conj} holds?  It could be that for general matrices, it is easier to analyze~\eq{variance-conj} rather than~\eq{weak-variance-conj} because the right hand side is in terms of the arithmetic mean, rather than the more complicated quadratic matrix products in~\eq{weak-variance-conj}.

\paragraph{Effect of biased orderings.}  Another possible technique for solving incremental algorithms is to choose the best ordering of the increments to reach the cost function.  In terms of matrices, can we find the ordering of the matrices $\mtx{A}_i$ that achieves the minimum norm.  At first glance this seems daunting.  Suppose $\mtx{A}_i=\vct{a}_i\vct{a}_i^T$ where the $\vct{a}_i$ are all unit vectors.  Then for $\sigma \in S_n$
\[
	\left\| \prod_{i=1}^n \mtx{A}_{\sigma(i)} \right\| = \prod_{i=1}^{n-1} | \langle \vct{a}_{\sigma(i)}, \vct{a}_{\sigma(i+1)}\rangle | 
\]
minimizing this expression with respect to $\sigma$ amounts to finding the minimum weight traveling salesman path in the graph with weights $\log | \langle \vct{a}_i, \vct{a}_j \rangle |$.  Are there simple heuristics that can get within a small constant of the optimal tour for these graphs?  How do greedy heuristics fare?  This sort of approach was explored with some success for the Kaczmarz method by~\citet{Needell11}.


\paragraph{Nonlinear extensions}

Extending even the random results in this paper to nonlinear algorithms such as the general incremental gradient descent algorithm or randomized coordinate descent would require modifying the analyses used here.  However, it would be of interest to see which of the randomization tools employed in this work can be extended to the nonlinear case.   For example, if we assume that the cost function~\eq{glob-prob} has summands which are sampled i.i.d., can we use similar tools (e.g., Jensen's inequality, moment bounds) to show that without-replacement sampling works even in the nonlinear case?

\section*{Acknowledgements}

The authors would like to thank Dimitri Bertsekas, Aram Harrow, Pablo Parrilo for many helpful conversations and suggestions.  BR is generously supported by ONR award N00014-11-1-0723 and NSF award CCF-1139953.  CR is generously supported by the Air Force Research Laboratory (AFRL) under prime contract no. FA8750-09-C-0181, the NSF CAREER award under IIS-1054009, ONR award N000141210041, and gifts or research awards from Google, Greenplum, Johnson Controls, Inc., LogicBlox, and Oracle.  Any opinions, findings, and conclusion or recommendations expressed in this work are those of the authors and do not necessarily reflect the views of any of the above sponsors including DARPA, AFRL, or the US government.

{\small
\bibliographystyle{abbrvnat}
\bibliography{/Users/brecht/Documents/LaTeX/brecht}
}

\appendix

\section{Additional calculations for random matrices}\label{sec:gaussian-calc}

For the special case of Wishart matrices, we can show that the gap between the norm of the arithmetic and geometric means in~\ref{eq:weak-variance-conj} is quite large.

\begin{lemma}
For $i=1,\dots,k$, we have
\[ \E_{\cal G}[ \mtx{V}( i,\dots,i) ] \geq \|x\|^2r 2^{-2k} \frac{4k!}{2k!} \sigma^{4k} \]
\end{lemma}

\begin{proof}
\begin{align*}
\E[ \mtx{V}(xi,\dots,i) ] & \geq 
\sum_{u=1}^{d} x_{u}^2 \E_{\cal G}[A^{(i)}_{u,u}A^{(i)}_{u,u} \prod_{j=2}^{k-1} A^{(i)}_{u,u} A^{(i)}_{u,u}] \\
& = \|x\|^2 \E[(A^{(i)}_{u,u})^{2k}] 
\end{align*}
The first inequality is because all the terms are positive and we are
selecting out only the self loops. The equality just groups terms. The
following lower bound completes the proof.
\begin{align*}
\E[ (A_{uu}^{i})^{2k} ]  
= \E\left[ \sum_{l_1,\dots,l_r} {2k \choose l_1,\dots,l_r} \prod_{l=1}^{r} g_{il}^2{l_i} \right] 
\geq  \sum_{l=1}^{r} \E[g_{il}^{4k}] = r 2^{-2k} \frac{4k!}{2k!} \sigma^{4k} \\
\end{align*}
\end{proof}

A simple corollary is the following lower bound on the arithmetic mean
\[ \E_{wr}\E[ \mtx{V}( i_1,\dots,i_k) ] \geq k^{-k} \norm{x}^2 r 2^{-2k}
\frac{4k!}{2k!} \sigma^{4k} \]

We examine the following ratio $\rho(r,k,d)$
\[ \rho(k,r,d) = \frac{\E_{wr}\E_{{\cal G}}[ \mtx{G}(x, i_1,\dots,i_k) ] }{\E_{wo}\E_{{\cal G}}[ \mtx{G}(x, i_1,\dots,i_k) ]} \geq r\frac{4k!}{2k!} (4kr(r+d+1))^{-k} \]

For fixed $r,d$, $\rho$ grows exponentially with $k$.
\begin{lemma} For $k,r,d \geq 0$ then
\[ \rho(k,r,d) \geq r e^{\frac{1}{4k(k+1)}}
\left( \frac{16k}{e^2 r (r+d+1)} \right)^{k} \]
\end{lemma}

\begin{proof}
We use a very crude lower and upper bound pair that holds for all
$k$~\citep[p.~55]{CLRBook}.

\[ \sqrt{2\pi k} \left(\frac{k}{e}\right)^{k} e^{\frac{1}{2k+1}} \leq k! \leq 
\sqrt{2 \pi k} \left(\frac{k}{e}\right)^{k} e^{1/2k}\]
With this inequality, we can write:
\begin{align*}
\rho(k,r,d) & \geq r \exp\{ k \ln (4k/e)^4 - k \ln (2k/e)^2 - k \ln (4kr(r+d+1)) + \frac{1}{2k} - \frac{1}{2k+1}  \} \\
& = r \exp \left\{ k \ln \frac{4^2 k}{e^2 r(r+d+1)} - \frac{1}{4k(k+1)} \right\} = r \left(\frac{16 k}{e^2 r(r+d+1)}\right)^{k} e^{\frac{1}{4k(k+1)}}
\end{align*}
\end{proof}

\section{Proof that harmonic frames satisfy the noncommutative arithmetic-geometric mean inequality}\label{sec:frames}

\paragraph*{Problem}
Let $S$ be the symmetrized geometric mean of a set of rank $1$, idempotent matrices that are
parametrized by angles $\phi_1,\dots,\phi_n$ (This is slightly more
general than we need for our theorem above). Our goal is to compute
the $2$-norm of $S$:

\[ \max_{v : \norm{v} = 1} v^{T}Sv  = \max_{\phi_v \in [0,2\pi]}
\frac{1}{n!} \sum_{\sigma \in S_n} \cos( \phi_v - \phi_{\sigma(1)}) \times
\prod_{i=1}^{n-1} \cos(\phi_{\sigma(i)} - \phi_{\sigma(i+1)}) \times
\cos(\phi_{\sigma(n)} - \phi_v)  \]

\subsection{Cosine combinatorics}

We will write this function as fourier transform (we pull out the
$2^{n}$ for convenience):
\begin{equation}
2^{n} \max_{v : \norm{v} = 1} v^{T}Sv =  \sum_{k} c_k \cos(k \pi n^{-1}) + \max_{\phi_v}  \sum_{l}d_l \cos(\phi_v + l \pi n^{-1}) 
\label{eq:cos:master}
\end{equation}

To find the $c_k$ and $d_k$, we repeatedly apply the following
identity:

\[ \cos x \cos y = 2^{-1}(\cos (x + y) +  \cos (x - y)) \]

Fix $\phi_1,\dots,\phi_2,\dots,\phi_n,\dots \in [0,2\pi]$. We first
consider a related form, $T_n$, for $n=1, 2,3,\dots,$ defined by the following recurrence
\[ T_1 = 1 \text{ and } T_{n+1} = T_{n} \cos(\psi_{n} - \psi_{n+1}) \]
We compute $T_n$ using the above transformation. But, first, we show
the pattern by example:

\begin{example}

\begin{align*}
T_1 & = 1\\
T_2 & = \cos(\psi_1 - \psi_2)\\
T_3 & = \cos(\psi_1 - \psi_3) + 
\cos(\psi_1 - 2\psi_2 + \psi_3)\\
T_4 & = \cos(\psi_1 - \psi_4) + 
\cos(\psi_1 - 2\psi_3 + \psi_4) + 
\cos(\psi_1 - 2\psi_2 + 2\psi_3 - \psi_4)
\cos(\psi_1 - 2\psi_2 + \psi_4)\\
T_4 & = \cos(\psi_1 - \psi_4) + 
\cos(\psi_1 - 2\psi_2 + 2\psi_3 - \psi_4) + 
\cos(\psi_1 - 2\psi_3 + \psi_4) + 
\cos(\psi_1 - 2\psi_2 + \psi_4)\\
T_5 & = \cos(\psi_1 - \psi_5) + 
\cos(\psi_1 - 2\psi_2 + 2\psi_3 - \psi_5) + 
\cos(\psi_1 - 2\psi_3 + \psi_5) + 
\cos(\psi_1 - 2\psi_2 + \psi_5)\\
& = 
\cos(\psi_1 - 2\psi_4 + \psi_5) + 
\cos(\psi_1 - 2\psi_2 + 2\psi_3 - 2\psi_4 + \psi_5) + 
\cos(\psi_1 - 2\psi_3 + 2 \psi_4 - \psi_5) \\
&\qquad\qquad\qquad + 
\cos(\psi_1 - 2\psi_2 + 2 \psi_4 - \psi_5)\\
\end{align*}

\noindent
In our computation above, $\psi_1 = \phi_v = \psi_n$. And so, after
writing this out, we will get two kinds of terms: even terms
(corresponding to $c_k$) that do not depend on $\psi_v$ (they cancel)
and odd terms that do contain $2 \psi_v$.
\end{example}

We encapsulate this example in a lemma:
\begin{lemma} With $T_n$ as defined above, we have for $n \geq 2$
\[ T_{n} = 2^{-n} \sum_{k} \sum_{\stackrel{i_1,\dots,i_k}{1 < i_1 < i_2 < \dots < i_k < n}} \cos(\psi_1 - 2 \psi_{i_1} + 2 \psi_{i_2} - \dots + (-1)^{k} 2\psi_{i_k} + (-1)^{k+1} \psi_{n}) \]
\end{lemma}

\begin{proof}
By induction, we have:
\begin{align*} 
T_{n+1} & = 2^{-n}\sum_{k} \sum_{\stackrel{i_1,\dots,i_k}{1 < i_1 < i_2 < \dots < i_k < n}} \cos(\psi_1 - 2 \psi_{i_1} + 2 \psi_{i_2} - \dots + (-1)^{k} \psi_{i_k} + (-1)^{k+1} \psi_{n})\cos(\psi_{n} - \psi_{n+1})\\
& = 2^{-(n+1)} \sum_{k} \sum_{\stackrel{i_1,\dots,i_k}{1 < i_1 < i_2 < \dots < i_k < n}} 
\cos(\psi_1 - 2 \psi_{i_1} + 2 \psi_{i_2} - \dots + (-1)^{k} \psi_{i_k} + (-1)^{k+2} \psi_{n+1})\\
 + & \cos(\psi_1 - 2 \psi_{i_1} + 2 \psi_{i_2} - \dots + (-1)^{k} \psi_{i_k} + 2(-1)^{k+1} \psi_{n} + (-1)^{k+2} \psi_{n+1})\\
& = 2^{-(n+1)} \sum_{k} \sum_{\stackrel{i_1,\dots,i_k}{1 < i_1 < i_2 < \dots < i_k < n+1}} 
\cos(\psi_1 - 2 \psi_{i_1} + 2 \psi_{i_2} - \dots + (-1)^{k} \psi_{i_k} + (-1)^{k+2} \psi_{n+1})
\end{align*}
\end{proof}

Fix an $n$. We now count a symmetrized version of $T_{n}$ defined as follows: For
$\sigma \in S_{n}$:
\[ S_{n} = \frac{1}{n!} \sum_{\sigma \in S_{n}} \prod_{i=1}^{n-1} \cos(\sigma(i) - \sigma(i+1)) \]

We now show that $S_{n}$ can be written in a form that removes the
permtuation. We also assume some structure here that mimics our
product above, namely that $\phi_{1} = \phi_{n}$.

\begin{lemma} Let $\phi_1,\dots,\phi_{n} \in [0,2\pi]$ such that $\phi_1 = \phi_{n}$. Then,
\begin{align*}
S_{n} & = \sum_{X,Y \subseteq [n] : |X| = |Y| \text{ or } |X| = |Y| + 1} {|X| + |Y| \choose |Y|}^{-1} \cos\left(2 \left(\sum_{i \in X} \phi_{i} - \sum_{j \in Y} \phi_{j}\right) \right)
\end{align*}
\end{lemma}

\begin{proof}
To see this formula, Consider a pair of sets $X,Y \subseteq [n]$. In
how many permutations $\sigma \in S_{n}$ does $(X,Y)$ contribute a
term? We need to choose $|X| + |Y|$ positions for these terms to
appear out of $n$ possible places in the order. Thus there are ${n
  \choose |X| + |Y|}$ permutations to choose the slots for $(X,Y)$. An
$(|X|,|Y|)$ pair only appears in a permutation in $\sigma$ if the
elements of $X$ and $Y$ can be alternated starting with $X$. This
implies that $|X| = 2|Y| + z_i$ where $z_{i} \in \set{0,1}$. Moreover,
there are $|X|!|Y|!(n-(|X|+|Y||)!$ permutations that respect this
structure (for any choice of $|X| + |Y|$ slots, any ordering of $X$
and $Y$ and the elements outside can occur). 
\[ {n \choose |X| + |Y|} |X|! |Y|! (n-(|X|+|Y|)! = n! {|X| + |Y| \choose |Y|}^{-1} \]
Pushing the $1/n!$ factor inside completes the proof.
\end{proof}

\subsection{Counting on harmonic, finite groups}

In the case we care about, the $\phi_i$ have more structure: the set
$\{ 2\phi_i\}_{i=1}^{n}$ forms a cyclic group under addition modulo
$2\pi$. Let $n$ denote the number of elements in the frame. Fix
$n$. Let $\zeta$ denote a $n$th root of unity. Define a (harmonic)
generating function $f$

\[  f(\zeta,y,z) = \prod_{i=1}^{n} (1 + \zeta^{i}x + \zeta^{-i}y)  \]
We give a shorthand for its coefficients $q_{k,m}$ and $r_{k,m}$ as follows
\[ q_{k,m} := [\zeta^{k}x^{m}y^{m}] f
 \text{ and }
  r_{k,m} := [\zeta^{k}x^{m+1}y^{m}] f
\]

We observe that $q_{k,m}$ computes the number of sets $(X,Y)$ where
$X,Y \subseteq \mathbb{Z}_n$ such that:
\begin{enumerate}
\item $\sum_{i \in X} i - \sum_{j \in Y} j = k \mod n$ (since we inspect $\zeta^{k}$), 
\item $|X| = |Y| = m$ (since we inpect $x^{m}y^{m}$), 
\end{enumerate}
For $r_{k,m}$ the only change is that $|X| = |Y| + 1$ (since
$x^{m+1}y^{m}$). With this notation, we can express the coefficients
from Eq.~\ref{eq:cos:master}.
\[ c_k = \sum_{m} q_{k,m}{2m \choose m}^{-1} \text{ and } d_k = \sum_{m} r_{k,m}{2m + 1 \choose m}^{-1} \]

We use this representation to prove that the symmetrized geometric mean is rotationally
invariant (i.e., $d_k = 0$ for $k=0,1,\dots,n-1$). First, we show that
all $d_k$ are equal.

\begin{lemma}
Consider a frame of size $n$. For any $m$ and $k,l = 0,\dots,n-1$,
$\sum_{k} d_{k} \cos(\phi_v + 2 \pi k) = 0$. 
\end{lemma}

\begin{proof} This follow by examining the generating function above.
First observe that we have congruence $f(x,x^{j}y,z) = f(x,y,z)$ for
$j=0,\dots,n-1$ tells us that $[x^{j}yz^{m}]f = [yz^{m}]f$. And, the
congruence that $[x^{j}yz^{m}]f = [x^{j}y^{-1}z^{m}]f$. Combining
these facts, we have that $r_{k,m} = r_{l,m}$. Since this holds for
all $k,l$, we can conclude that $d_{k} = d_{l}$ by summing over
$m$. Finally, since $\sum_{l=0}^{n} \cos(\phi_v + 2 l \pi n^{-1}) = 0$
for any fixed $\phi_v$ we conclude the lemma.
\end{proof}

Since the symmetrized geometric mean does not depend on $\phi_v$, we conclude it must be of
the form $\alpha I$ for some $\alpha$. The remainder of this note is
to compute that $\alpha$.

\subsection{Computing the coefficients}

The argument of this subsection is a generalization of that of~\citet{Konvalina95}.

\begin{lemma}
\[ q_{k,m} = (-1)^{k} {n - k \choose k} \frac{n}{n-k} \]
\end{lemma}

\begin{proof}
Define $R_n$ as:
\[ R_n(x,y) = x^{n} + (-y)^{n} - \sum_{k} {n - k \choose k} \frac{n}{n-k} (xy)^{k}  \]

Since $f(\zeta,x,y) = f(\zeta^{i},x,y)$ for any integer $n$, $F_n(x,y)
= f(\zeta,x,y)$ is a function of $n$ alone. That is, we can write
\[ F_n(x,y) = \prod_{\zeta \in U_n} (1 + \zeta x + \zeta^{-i} y) \]
Thus, claim boils down to
$F_n(x,-y) = R_n(x,y)$.

We show that the zero sets of $F_n(x,-y)$ and $R_n$ are equal. The
zero set of $F_n(x,-y)$ is the set of lines described by
\[ \setof{(x,y)}{y = \zeta + \zeta^2 x} \text{ for } \zeta \in U_n \]
where $\zeta$ is any $n$-th root of unity.  Substituting $y$ at the
root equation, we get that $xy = x\zeta + \zeta^2x^2$.

Now, we check that the following is zero:
\[ x^{n} + (-y)^{n} -  \sum_{k} {n - k \choose k} \frac{n}{n-k} (\zeta x + \zeta^2 x^2)^k
\]

Here, we use the generating function:
\[ \sum_{k} \frac{n}{n-k}{n-k \choose k} y^{k} = \left(\frac{1-\sqrt{1+4y}}{2}\right)^n + \left(\frac{1+\sqrt{1+4y}}{2}\right)^n\,.\] 
Using this sum, we have:
\begin{align*}
& \sum_{k=0}^{n} {n -k \choose k} \frac{n}{n-k} (\zeta x + \zeta^2x^2)^{k} \\
= &  \left(\frac{1 - (2\zeta x + 1)}{2}\right)^n + \left(\frac{1 + (2\zeta x + 1)}{2}\right)^n \\
= & (\zeta x)^{n}  + (1+\zeta x)^{n}\\
= & x^{n} + (-y)^{n}  
\end{align*}
The first equality follows from $1 + 4 \zeta x + 4 \zeta^2 x^2 =
(2x\zeta + 1)^2$. The second is just algebra. Finaly, we use on each
term that $\zeta^{n} = 1$ and that $\zeta + \zeta^2 x = -y$. This claim holds for all $\zeta$
that are roots of unity, and so the function is identically zero.

To conclude the proof, observe that the zero set described above is
the union of $n$ lines of the form $(1 + \zeta x + \zeta^{-1}
y)$. These lines are unique in $\mathbb{C}$: if $(1 + \zeta x +
\zeta^{-1}y) = (1 + \omega x + \omega^{-1}y)$ then since the $x$
coefficients are the same $\zeta = \omega$ and so they must be the
same. By direct inspection, this $R_n$ can only have these factors
(else the total degree would be higher). Hence, $R_n$ = $Q_n$.
 \end{proof}

\subsection{Finally, to a hypergeometric series}

It is possible to get an explicit formula for $\lambda$ that is
related to $_3F_2$. We consider the following series and show that it
is hypergeometric in $k$:
\[ \sum_k T(n,k) {2k\choose k}^{-1} x^{k} = \sum_{k} v(k) x^{k} \] 
Consider the ratio:
\begin{align*}
\frac{v(k+1)}{v(k)} & =
- \frac{ (n-k-1)!((k+1)!)^2}{k+1!(n-2k-2)! (n-k-1) (2k+2)! } \frac{2k!(n-2k)!k! (n-k)}{ (n-k)! (k!)^2}\\
&= - \frac{(n-2k)(n-2k-1)(k+1)}{(n-k-1)(2k+2)(2k+1)}\\
&= \frac{(k-n/2)(k - n/2+1/2)(k+1)}{(k-n+1))(k+1/2)(k+1)}\\
\end{align*}
And so, this is a hypergeometric:
\[ _2F_3\left[\begin{array}{ccc}
1 & -n/2+1/2 & -n/2\\
1/2 & -n+1 
  \end{array}; 1 \right] = \mathcal{O}(n^{-1}) \]

This completes the proof.

\end{document}